\documentclass[12pt]{article}
\usepackage[latin1]{inputenc}
\usepackage[english]{babel}
\newtheorem{theorem}{Theorem}
\newtheorem{lemma}{Lemma}
\newtheorem{definition}{Definition}
\newtheorem{proof}{Proof}
\usepackage{stmaryrd}
\usepackage{enumerate}
\usepackage[noend]{algpseudocode}
\usepackage{calrsfs}
\usepackage{graphicx}
\usepackage{amssymb}
\usepackage{amsmath}
\usepackage{epstopdf}
\usepackage{epsfig}
\usepackage{subcaption}
\usepackage{pstricks}
\usepackage{fancyheadings}
\usepackage{pgfplots}
\usepackage{physics}
\usepackage{stackengine}
\usepackage[margin=1.0in]{geometry}

\usepackage{latexsym}
\DeclareMathAlphabet{\pazocal}{OMS}{zplm}{m}{n}

\newtheorem{corollary}{Corollary}
\newtheorem{remark}{Remark}

\newcommand{\lJump}{[\![}
\newcommand{\rJump}{]\!]}

\newcommand{\pdd}[2]{\frac{\partial #1}{\partial #2}}

\newcommand{\bsprod}[2]{\left\langle #1,#2\right\rangle}

\newcommand{\en}[1]{{\left\vert\kern-0.25ex\left\vert\kern-0.25ex\left\vert #1 
    \right\vert\kern-0.25ex\r6ight\vert\kern-0.25ex\right\vert}}

\makeatletter
\def\BState{\State\hskip-\ALG@thistlm}
\makeatother

\begin{document}
\title{On energy-stable and high order finite element methods for the wave equation in heterogeneous media with perfectly matched layers}
\author{Gustav  Ludvigsson\thanks{Division of Scientific Computing, Department of Information Technology, Uppsala University, Sweden.}, \and Kenneth Duru\thanks{Mathematical Sciences Institute, The Australian National University, Australia.}, \and Gunilla Kreiss \thanks{Division of Scientific Computing, Department of Information Technology, Uppsala University, Sweden.} }
\pagenumbering{arabic}
\maketitle
\begin{abstract}
 This paper presents a stable finite element approximation for the acoustic wave equation on second-order form, with perfectly matched layers (PML) at the boundaries. Energy estimates are derived for varying PML damping for both the discrete and the continuous case. Moreover, a priori error estimates are derived for constant PML damping. Most of the analysis is performed in Laplace space. Numerical experiments in physical space validate the theoretical results
\end{abstract}

\section{Introduction\label{sec:introduction}}
Accurate and effective domain truncation schemes are necessary for efficient numerical simulation of wave propagation problems in many applications. This is because waves typically propagate over large domains, while the domain of interest usually is much smaller. Solving the problem over the whole domain may be computationally costly or even impossible. It is necessary to truncate the large domain so that the problem only has to be solved over the smaller domain of interest.
When truncating the domain, an artificial boundary is introduced, and careful treatment of this artificial boundary is necessary to get 
reliable results from numerical simulations.

There are now two powerful and equally appealing approaches for realizing an effective domain truncation scheme. The first is to use a non-reflecting boundary condition (NRBC), \cite{collino1993high,hagstrom2010radiation,givoli2004high}, which is a boundary condition that is enforced at the artificial boundary such that spurious reflections from the artificial boundary are sufficiently minimized. The second one is adding an absorbing layer where the underlying equations are transformed such that waves decay rapidly as they propagate in the layer. The perfectly matched layer (PML), \cite{berenger1994,duru2016role,DuKr}, is an absorbing layer with the desirable property that all waves enter the layer and are absorbed without any reflections, regardless of frequency and angle of incidence. 

The PML technology was initially introduced for electromagnetic waves by B{\'e}renger in \cite{berenger1994}. 
The approach has since then been extended to other wave propagation problems such as elastic wave propagation, see for example \cite{appelo2006new,ElasticDG_PML2019,DuKrSIAM}, and acoustic wave propagation, for example, \cite{grote2010efficient,abarbanel1999well}. Please see \cite{https://doi.org/10.48550/arxiv.2201.03733} for a recent review. In this paper, we consider the PML formulation for the second-order scalar wave equation presented in \cite{grote2010efficient}.

The PML approach has grown popular over the last decades due to its versatility,  efficiency, and ease of implementation using standard numerical methods such as finite difference, finite volume, finite element, spectral, and discontinuous Galerkin methods. One caveat is that in practice, once the PML-formulation is discretized, the PML approach is no longer completely non-reflecting. This is due to the discretization, which disturbs the perfect matching, and the finite PML width, which means that the residual outgoing waves will reflect from the outer boundary and travel back through the layer. 

However, stability for the chosen continuous PML and the corresponding discretization is imperative for the approach to be efficient and robust and for results from simulations to be reliable. There is a large body of work treating the stability and well-posedness of PML formulations at the continuous PDE level.
For instance, in \cite{BecaheKachanovska2021,DIAZ20063820} the Cagniard--De Hoop method was used to derive analytical solutions of the continuous PML for the 2D and 3D acoustic wave equation in a homogeneous medium.
We note that the stability analysis of the PML at the discrete level has attracted less attention in the literature. This is attributable to the complex and asymmetrical forms of the PML equations, which make applying standard analytical tools such as the energy methods technically difficult.

The stability of a PML model has been studied from a mode perspective. In \cite{becache2003stability} a necessary geometric stability condition was formulated for the initial value problems of general first-order hyperbolic systems. When violated, exponentially growing modes in time are present, making the layer useless. However, even if the Cauchy problem does not have growing modes, there can still be stability issues related to boundaries and interfaces. Mode analysis has to some degree been extended to include such issues, see \cite{DuKrSIAM,DuKr,duru2015boundary,duru2012accuracy,DURU2014757}. However, the mode-based stability results are also cumbersome to extend to discrete settings.

A standard approach to the stability of a numerical scheme is to first derive an energy estimate for the corresponding continuous problem and then mimic the energy analysis in the discrete setting. In physical space, there are severe technical difficulties in deriving energy estimates for problems with PMLs, especially when including boundaries, heterogeneous medium parameters, and allowing for spatially varying damping parameters. In \cite{appelo2006perfectly} first-order constant coefficient Cauchy problems are considered. A temporally decaying energy in physical space could be found by requiring several algebraic inequalities, which must be verified in Fourier-Laplace space. However, the energy involves combinations of higher-order derivatives of the unknowns and is therefore difficult to work within a discrete setting. An important step was recently taken in  \cite{baffet2019energy}, where energy estimates, valid in both continuous and discrete settings, are derived in physical space for the PML formulation for the second-order wave equation presented in \cite{grote2010efficient}, but only for constant media and damping functions. 

This paper aims to formulate a stable and accurate finite element method (FEM) for the wave equation on a domain surrounded by a PML. We take the approach of deriving an energy estimate in Laplace space, as in \cite{duru2019energy}, where stability results are derived for a discontinuous Galerkin spectral element method (DGSEM) for the wave equation as a first-order system. The analysis in \cite{duru2019energy} allows constant damping parameters, while the stability analysis in this paper extends the analysis to also include continuously variable damping coefficients. The key step taken in  \cite{duru2019energy} was to construct physics-based numerical fluxes with appropriate penalty weights, which allowed the derivation of energy estimates in the Laplace space. For continuous Galerkin FEM, we show that the choice of the weak form and solution space are crucial for proving numerical stability and convergence of high order FEM for the PML for the wave equation in general domains. The chosen weak form allows us to prove energy estimates for heterogeneous medium parameters with piecewise continuous wave speed and continuously varying damping functions, both in continuous and discrete settings. In addition, we prove a convergence result. These are the main results of this paper. Also worth noting is that the paper demonstrates the strength of the strategy of doing energy analysis in Laplace space. We believe that this strategy can facilitate stability results for many other PML problems.

The outline of this paper is as follows. In Section \ref{sec:equations} we present the underlying wave problem in three space dimensions. In Section \ref{sec:pml} we introduce the PML formulation. In Section \ref{sec:laplace} we derive continuous energy estimates. In Section \ref{sec:fem} the weak form is presented. Numerical stability and a priori error estimates are derived in Section \ref{sec:discrete}. The numerical experiments are presented in Section \ref{sec:experiments}. Finally, in Section \ref{sec:discussion} we summarise the results and draw conclusions.
\section{Equations \label{sec:equations}}
We consider the time dependent acoustic wave equation on second order form in a source free heterogeneous medium
\begin{equation} \label{eq:wave_eq}
\frac{1}{\kappa}\ddot{u}   - \div\left(\frac{1}{\rho} \grad {u}\right) = 0,  \quad  \bold{x}\in \Omega,   \quad t\in(0,T), 
\end{equation}
where $\bold{x} = (x,y,z) \in \Omega \subset \mathbb{R}^3$, $u:\Omega \times (0,T) \rightarrow \mathbb{R}$ is the solution to the wave equation \eqref{eq:wave_eq}, and $T>0$ is the final time. The medium is defined by the wave speed $c = \sqrt{\kappa/\rho} $,  where  $\kappa >0$ is the bulk-modulus and $\rho >0$ is the density.  
The wave equation \eqref{eq:wave_eq} is subjected to  the initial conditions
\begin{equation} \label{eq:u0}
u(\bold{x},t = 0) = g_1(\bold{x}), \quad \dot{u}(\bold{x},t = 0) = g_2(\bold{x}), \quad \bold{x} \in  \Omega,
\end{equation}
where we assume that the initial data is compactly supported.
The domain $\Omega$ is bounded by the boundary $\partial \Omega$. On $\partial \Omega$ we augment the wave equation with the well-posed  linear boundary conditions
\begin{equation} \label{eq:general_bc}
\frac{1-r}{2} Z\dot{u} + \frac{1+r}{2}\grad u \cdot \mathbf{n} = 0, \quad \bold{x} \in \partial \Omega,
\end{equation}
where $\mathbf{n}$ is the outward unit normal on the boundary $\partial \Omega$, $r$ with $|r| \le 1$ is the reflection coefficient and $Z = \rho c >0$ is the wave impedance.
Note that $r=-1$  corresponds to a Dirichlet boundary condition
\begin{equation} \label{eq:dirichlet}
{u}=0 \implies \dot{u} = 0, \quad \bold{x} \in \partial \Omega,
\end{equation}
  $r=1$  yields a Neumann boundary condition
\begin{equation} \label{eq:neumann}
\grad{u} \cdot \mathbf{n} = 0, \quad \bold{x} \in \partial \Omega,
\end{equation}
and $r=0$  corresponds to the classical first order absorbing boundary condition
\begin{equation} \label{eq:neumann}
\frac{1}{2} Z\dot{u} + \frac{1}{2}\grad u \cdot \mathbf{n} = 0, \quad \bold{x} \in \partial \Omega.
\end{equation}
\begin{remark}
We have tacitly considered homogeneous boundary data under the usual assumptions of compatibility between boundary and initial data. The analysis follows through for the case of non-homogeneous boundary data, but this would complicate the algebra.
\end{remark}

We will now derive the standard energy estimate, on which a proof of well-posedness can be based. 
To begin, for real functions $u,v$ defined over $\Omega$ we introduce the inner products
$$
\left(u,v\right) = \int_{\Omega} u^*v d \bold{x},
$$
and for the boundary
$$
\bsprod{u}{v} =\int_{\partial \Omega}  u^*v dS
$$
together with the corresponding norm
$$
\|u\|^2_{b} = (b u,u), \quad b(\mathbf{x})>0.
$$
If $b \equiv 1$ we will omit the subscript $b$.
We define the energy as
\begin{equation} \label{eq:energy}
E(t)  = \frac{1}{2}\left(\|\dot{u} \|^2_{1/\kappa} +  \|\grad u \|^2_{1/\rho} \right).
\end{equation}

To begin with assume that the material parameters are continuous and there is a sufficiently smooth solution. We will now use the standard technique to derive an energy estimate.
Multiply equation \eqref{eq:wave_eq} by
$\dot{u}$ and integrate by parts in space to get
\begin{equation}
\left(\frac{1}{\kappa}\ddot{u},\dot{u}\right)  + \left(\left(\frac{1}{\rho} \grad {u}\right), \grad {\dot{u}}\right) - \bsprod{\frac{1}{\rho} \left( \grad {u} \cdot \bold{n}\right)}{\dot{u}} = 0.
\end{equation}
The boundary term vanishes for the cases when $|r| = 1$, which corresponds to either the Dirichlet boundary condition \eqref{eq:dirichlet} or the Neumann boundary condition \eqref{eq:neumann}. If $r\ne -1$ we have
\begin{equation}
\bsprod{\frac{1}{\rho} \left( \grad {u} \cdot \bold{n}\right)}{\dot{u}} = \bsprod{c\frac{1-r}{1+r}\dot{u}}{ \dot{u}} ,
\end{equation}
where $c>0$ is the wave speed. It follows that
\begin{equation}
\frac{1}{2} \frac{d}{dt}\left(\|\dot{u} \|^2_{1/\kappa} + \|\grad  u \|^2_{1/\rho} \right) + BT= 0, 
\end{equation}
where
\begin{align}
    BT = \begin{cases}
0,\qquad r =-1\\
\bsprod{c\frac{1-r}{1+r}\dot{u}}{ \dot{u}},\quad r \ne -1.
\end{cases}
\end{align}
In particular we note that $BT \ge 0$ since $|r|\le1$. Thus
\begin{equation}
\frac{d}{dt} E(t) = -BT \le 0,
\end{equation}
which after integration in time gives us the energy estimate
\begin{equation}
E(t) \le E(0) \quad \forall t \geq 0.
\end{equation}
If $|r|=1$ then we have energy conservation, that is $E(t) = E(0)$ for all $t\ge 0$.
If $|r|<1$ then energy is lost through the boundaries.

We can include piece-wise continuous material properties in the above. At interfaces $\Gamma$ where the wave speed $c>0$ is discontinuous we prescribe the interface condition
\begin{equation}\label{eq:interface_condition}
{\left\lJump  {u}  \right\rJump } =0, \quad {\lJump\frac{1}{\rho}\grad{u}\cdot \mathbf{n}\rJump } =0,
\end{equation}
where $\mathbf{n}$ 

is a unit normal at the interface.
In the derivation of the energy estimate interface terms appear, but vanish due to this condition, and the same energy estimate holds also for this case. In the following analysis we will include the possibility of material interfaces.
\section{PML\label{sec:pml}}

In this section we re-derive the PML problem used in \cite{grote2010efficient} using the well-known complex coordinate stretching technique \cite{chew19943d}.

We begin with introducing the Laplace transform in time. Let $\widehat{u} = \mathcal{L}(u)$ be the Laplace transform of $u$ defined by
\begin{equation}\label{eq:laplace}
\widehat{u}(x,y,z,s) = \mathcal{L}(u):= \int_0^\infty e^{-st}u(x,y,z,t) dt,
\end{equation}
where $s\in\mathbb{C}$ with $Re(s) = a>0$.
Laplace transforming the  wave equation \eqref{eq:wave_eq} in time, with homogeneous initial conditions, yields
\begin{equation}\label{eq_laplace}
\frac{s^2}{\kappa}\widehat{u}   - \div\left(\frac{1}{\rho} \grad {\widehat{u}}\right) = 0  
\end{equation}

We introduce the PML coordinate transformation
\begin{equation}\label{eq:ct}
\eta \rightarrow \widetilde{\eta} = \eta + \frac{1}{s} \int_0^\eta d_\eta(\theta) d\theta,\quad\eta=x,y,z,  
\end{equation}
where $d_\eta(\eta)$ are the damping functions, and $s$ is dual Laplace time variable.
Note that 

\begin{equation}\label{eq:laplace_s}
\pdd{\ }{\widetilde{\eta}} = \frac{1}{S_\eta} \pdd{\ }{\eta}, \quad 
S_\eta = 1 + \frac{d_\eta \left(\eta\right)}{s}.
\end{equation}
In the Laplace space we apply the coordinate transformation to the wave equation  \eqref{eq_laplace} yielding
\begin{equation} \label{eq:reduced_laplace}
\frac{s^2}{\kappa}\widehat{{u}} = {\mathbf{\nabla}}_d\cdot \left( \frac{1}{\rho}\grad_d \widehat{u}\right),
\end{equation}

where
$\grad_d = (\frac{1}{S_x}\pdd{ }{x},
\frac{1}{S_y}\pdd{ }{y},\frac{1}{S_z}\pdd{ }{z})^T$.

To invert the Laplace transforms we multiply \eqref{eq:reduced_laplace} with $S_x S_y S_z$, we have
\begin{equation}\label{eq:reduced_laplace_simple}
\frac{s^2 S_x S_yS_z}{\kappa}\widehat{u}   - \div\left(\mathbf{S}  \frac{1}{\rho} \grad_d {\widehat{u}}\right) = 0, \quad
\mathbf{S} = 
\begin{pmatrix}
{S_y S_z} & 0 & 0\\
0 & {S_x S_z} & 0\\
0 & 0 & {S_x S_y}
\end{pmatrix}.
\end{equation}

Introducing the diagonal matrices
$$
\mathbf{d} = 
\begin{pmatrix}
d_x & 0 & 0\\
0 & d_y & 0\\
0 & 0 & d_z
\end{pmatrix}, \quad 
\boldsymbol{\Gamma} = 
\begin{pmatrix}
d_y+d_z-d_x & 0 & 0\\
0 & d_x + d_z-d_y & 0\\
0 & 0 & d_x + d_y-d_z
\end{pmatrix},
$$
$$
\boldsymbol{\Upsilon} = 
\begin{pmatrix}
d_yd_z & 0 & 0\\
0 & d_xd_z & 0\\
0 & 0 & d_xd_y
\end{pmatrix},
$$
 the auxiliary variables $\widehat{\boldsymbol{\phi}}=\left(\widehat{\phi}^x,\widehat{\phi}^y, \widehat{\phi}^z\right)^T$, $\widehat{\psi}$ with 

\begin{subequations}
\label{eq:auxiliary_variables_laplace}
\begin{alignat}{2}
s\widehat{\boldsymbol{\phi}} &= -\mathbf{d} \widehat{\boldsymbol{\phi}}+\left(\frac{1}{\rho}\boldsymbol{\Gamma} + \frac{1}{s\rho} \boldsymbol{\Upsilon}\right) \grad{\widehat{u}},\\
s\widehat{\psi} & = \widehat{u},
\end{alignat}
\end{subequations}
and inverting the Laplace transforms results in the PML equation in the time-domain

\begin{subequations}
\label{eq:pml_system_3D}
\begin{alignat}{3}
\frac{1}{\kappa}\ddot{u} + \frac{1}{\kappa}\sum_{\eta} d_\eta\dot{u} + \frac{d_xd_yd_z}{\kappa}\sum_{\eta}\frac{1}{d_\eta} u  
&=  
\div\left(\frac{1}{\rho} \grad {u}\right) + \div{\boldsymbol{\phi}}  -\frac{d_xd_yd_z}{\kappa} \psi ,\label{eq:pml_1_3D} \\ 
  \dot{\boldsymbol{\phi}} &= -\mathbf{d} \boldsymbol{\phi}+\frac{1}{\rho}\boldsymbol{\Gamma} \grad{u}+\frac{1}{\rho} \boldsymbol{\Upsilon} \grad{\psi}, 
  \label{eq:pml_2_3D}
  \\
		\dot{\psi} &= u.\label{eq:pml_3_3D}
\end{alignat}
\end{subequations}

We will use boundary and interface conditions, which in Laplace space have the form

\begin{align}\label{eq:rotated_bc_0}
\frac{1-r}{2}c\left(\left(\mathbf{S}\mathbf{n}\right)\cdot \mathbf{n}\right) \left(s\widehat{u}\right) + \frac{r+1}{2}\left(\frac{1}{\rho} \mathbf{S}\grad_d \widehat{u}\right)\cdot \mathbf{n} =0,
\end{align}
and
\begin{align}\label{eq:rotated_IC_0}
{\left\lJump  \widehat{u}  \right\rJump } =0, \quad   {\lJump\left(\frac{1}{\rho} \mathbf{S}\grad_d \widehat{u}\right)\cdot \mathbf{n}\rJump } =0.
\end{align}
{Note that \eqref{eq:rotated_bc_0}  can be written as
\begin{align}\label{eq:rotated_bc_1}
S_xS_yS_z\left(\frac{1-r}{2}c\left(\sum_{\eta}\frac{n_\eta^2}{S_{\eta}}\right) \left(s\widehat{u}\right) + \frac{r+1}{2}\frac{1}{\rho} \sum_{\eta}\frac{n_\eta}{S_{\eta}^2}\frac{\partial \widehat{u}}{\partial \eta}\right) =0,
\end{align}
and since $S_xS_yS_z \ne 0$ we must also have
\begin{align}\label{eq:rotated_bc_2}
\left(\frac{1-r}{2}c\left(\sum_{\eta}\frac{n_\eta^2}{S_{\eta}}\right) \left(s\widehat{u}\right) + \frac{r+1}{2}\frac{1}{\rho} \sum_{\eta}\frac{n_\eta}{S_{\eta}^2}\frac{\partial \widehat{u}}{\partial \eta}\right) =0,
\end{align}
}
Similarly the second part of \eqref{eq:rotated_IC_0}  can be written as
\begin{align}\label{eq:rotated_IC_1}
{\lJump\left(\frac{1}{\rho} 
S_xS_yS_z 
\sum_{\eta}\frac{n_\eta}{S_{\eta}^2}\frac{\partial \widehat{u}}{\partial \eta}\right)
\rJump } =0.
\end{align}
Corresponding to \eqref{eq:rotated_bc_0} and \eqref{eq:rotated_IC_0} we have  relations in physical space  in terms of the solution and the auxilary variables. The corresponding boundary  condition is
\begin{equation}\label{eq:general_bc_pml}
c\frac{(1-r)}{2} \left(\dot{u} + \boldsymbol{\theta} \cdot \mathbf{n}\right) + \frac{(1+r)}{2}  \left(\frac{1}{\rho}\grad{u} + \boldsymbol{\phi}\right)\cdot \mathbf{n} =0, \quad (x, y, z) \in \partial{\Omega},
\end{equation}

and the corresponding interface conditions are
\begin{equation}\label{eq:interface_condition_pml}
{\left\lJump  {u}  \right\rJump } =0, \quad   {\lJump\left(\frac{1}{\rho}\grad{u} + \boldsymbol{\phi} \right)\cdot \mathbf{n}\rJump }=0,
\end{equation}
where

$$
\boldsymbol{\theta} 
= \begin{pmatrix}
\left((d_y+d_z) u + d_yd_z \psi\right)n_x\\
\left((d_x+d_z) u + d_xd_z \psi\right)n_y\\
\left((d_x+d_y) u + d_xd_y \psi\right)n_z
\end{pmatrix}.
$$
The conditions \eqref{eq:general_bc_pml} and \eqref{eq:interface_condition_pml} are non-trivial modifications of \eqref{eq:general_bc} and \eqref{eq:interface_condition} respectively.
 The modifications are chosen to ensure the energy estimate, and to allow for energy dissipation at the boundary.
Note that $r=-1$  corresponds to the Dirichlet boundary condition
$$
  {u} =0 \implies \dot{u} + \boldsymbol{\theta} \cdot \mathbf{n} =0, \quad (x, y, z) \in \partial{\Omega},
$$
and $r=1$  corresponds to the Neumann boundary condition
$$
 \left(\frac{1}{\rho}\grad{u}+ \boldsymbol{\phi}\right)\cdot \mathbf{n} =0, \quad (x, y, z) \in \partial{\Omega}.
$$

\section{Energy analysis for the PML
problem}\label{sec:laplace}
In this section we will derive the energy estimate for the PML problem in Laplace space.
Consider the inhomogeneaus system

\begin{subequations}
\label{eq:pml_system_3D_source}
\begin{alignat}{5}
\frac{1}{\kappa}\ddot{u} + \frac{1}{\kappa}\sum_{\eta} d_\eta\dot{u} + \frac{d_xd_yd_z}{\kappa}\sum_{\eta}\frac{1}{d_\eta} u  
&=  
\div\left(\frac{1}{\rho} \grad {u}\right) + \div{\boldsymbol{\phi}}  -\frac{d_xd_yd_z}{\kappa} \psi + \frac{1}{\kappa}F_{u},\label{eq:pml_1_3D_source} \\ 
  \dot{\boldsymbol{\phi}} &= -\mathbf{d} \boldsymbol{\phi}+\frac{1}{\rho}\boldsymbol{\Gamma} \grad{u}+\frac{1}{\rho} \boldsymbol{\Upsilon} \grad{\psi}+ \mathbf{f}_{\phi}, 
  \label{eq:pml_2_3D_source}
  \\
		\dot{\psi} &= u+ f_{\psi},
		\label{eq:pml_3_3D_source}
\end{alignat}
\end{subequations}
with vanishing initial data and subject to the boundary and interface conditions \eqref{eq:general_bc_pml} and \eqref{eq:interface_condition_pml}. Problems with homogeneous initial and boundary data can be reformulated as \eqref{eq:pml_system_3D_source} by a suitable change of variables.

Next we Laplace transform \eqref{eq:pml_system_3D_source} in time and get
\begin{subequations}
\label{eq:pml_system_3D}
\begin{alignat}{5}
\frac{s^2}{\kappa} S_xS_yS_z \widehat{u}  
&=  
\div\left(\frac{1}{\rho} \grad {\widehat{u}}\right) + \div{\widehat{\boldsymbol{\phi}}}  %
+ \frac{1}{\kappa}\left(\widehat{F}_u + \frac{d_xd_yd_z}{s}  \widehat{f}_\psi\right), 
\label{eq:pml_1_3D_laplace_pml_source}
\\
  s\widehat{\boldsymbol{\phi}} &= -\mathbf{d} \widehat{\boldsymbol{\phi}}+\frac{1}{\rho}\boldsymbol{\Gamma} \grad{\widehat{u}}+\frac{1}{\rho} \boldsymbol{\Upsilon} \grad{\widehat{\psi}} + \widehat{\mathbf{f}}_{\phi}, 
  \label{eq:pml_2_3D_laplace_pml_source}
  \\
		  s\widehat{{\psi}} &= \widehat{u} + \widehat{f}_\psi, \label{eq:pml_3_3D_laplace_pml_source}
\end{alignat}
\end{subequations}

Using \eqref{eq:pml_2_3D_laplace_pml_source}-\eqref{eq:pml_3_3D_laplace_pml_source} and  eliminating the auxiliary variables $\widehat{\boldsymbol{\phi}}, \widehat{{\psi}}$ in \eqref{eq:pml_1_3D_laplace_pml_source} yields
\begin{equation} \label{eq:reduced_laplace_forcing}
\frac{s^2 S_x S_yS_z}{\kappa}\widehat{u}   - \div\left(\mathbf{S}  \frac{1}{\rho} \grad_d {\widehat{u}}\right) = \frac{S_xS_yS_z}{\kappa}\widehat{F} + S_xS_yS_z\div\widehat{\mathbf{f}},
\end{equation}
with the boundary condition

\begin{align}\label{eq:general_bc_laplace_pml_2}
S_x S_yS_z\left(\frac{1-r}{2}c\left(\sum_{\eta}\frac{n_\eta^2}{S_{\eta}}\right) \left(s\widehat{u}\right) + \frac{r+1}{2}\frac{1}{\rho} \sum_{\eta}\frac{n_\eta}{S_{\eta}^2}\frac{\partial \widehat{u}}{\partial \eta}\right) =0,
\end{align}
and the interface condition

\begin{equation} \label{eq:interface_condition_laplace_pml_2}
{\left\lJump  \widehat{u}  \right\rJump } =0, \quad    {\lJump\left(\frac{1}{\rho} \mathbf{S}\grad_d \widehat{u}\right)\cdot \mathbf{n}\rJump } =0.
\end{equation}
Here
$$
\widehat{F} = \frac{1}{S_xS_yS_z}\left(\widehat{F}_u + \frac{d_xd_yd_z}{s}  \widehat{f}_\psi\right),
$$
$$
\quad S_xS_yS_z\div\widehat{\mathbf{f}} =  \div\widehat{\mathbf{f}}_g,
$$
and
$$
\widehat{\mathbf{f}}_g =\left(sI + \mathbf{d}\right)^{-1} \left(\frac{1}{s\rho} \boldsymbol{\Upsilon}\grad{\widehat{f}_\psi} + \widehat{\mathbf{f}}_{\phi}\right).
$$
\begin{remark}
In realistic settings we typically have $\mathbf{f}_{\phi}=f_{\psi} = 0$. However, for generality, we choose to include them in the  analysis in this section.
\end{remark}

In \cite{duru2019energy}, an acoustic first-order system with a surrounding PML is considered. After eliminating the velocity field the pressure is shown to satisfy \eqref{eq:reduced_laplace_forcing}-\eqref{eq:interface_condition_laplace_pml_2}. This means that we can take advantage of the analysis there.
Following \cite{duru2019energy}, we note that for $s =a+ib\neq 0$, $a>0,$ and $d_{\eta} \geq 0$, with $ \eta = x,y,z$, we have
\begin{equation}\label{eq:real}
\textrm{Re}\left(\frac{(sS_\eta)^*}{S_{\eta}}\right) = a + k_{\eta}, \qquad k_{\eta} = \frac{2d_{\eta}b^2}{\lvert sS_{\eta}\rvert ^2}\geq 0.
\end{equation}
Next define an energy for the solution in Laplace space
\begin{equation}\label{eq:Eu}
\begin{split}
E_{\widehat{u}}^2(s) &= \left(\frac{1}{\kappa} s\widehat{u},s\widehat{u}\right) + \sum_{\eta = x,y,z}\left(\frac{1}{\rho}\frac{1}{S_{\eta}} \pdd{\widehat{u}}{\eta},\frac{1}{S_{\eta}}\pdd{\widehat{u}}{\eta}\right) \\
 & = \left\|s\widehat{u}\right\|_{1/\kappa}^2 +\sum_{\eta = x,y,z}\left\|\frac{1}{S_{\eta}}\pdd{\widehat{u}}{\eta}\right\|^2_{1/\rho}\geq0,
\end{split}
\end{equation}

and an energy for the data
\begin{equation}\label{eq:Ef}
\begin{split}
E_f^2(s) &= \left(\frac{1}{\kappa}\widehat{F},\widehat{F}\right) + (\rho\mathbf{\widehat{f}},\mathbf{\widehat{f}})  \\
 & = \left\|\widehat{F}\right\|_{1/\kappa}^2+\left\|\mathbf{\widehat{f}}\right\|_{\rho}^2 >0,
\end{split}
\end{equation}
with 
$$
\mathbf{\widehat{f}} = \left({sS_y}\widehat{f}_x, {sS_y}\widehat{f}_y, {sS_z}\widehat{f}_z\right)^T.
$$

The following theorem gives an energy estimate in Laplace space. It is similar to Theorem 1 in \cite{duru2019energy}, but more general  since it allows continuously varying damping.

\begin{theorem}\label{thm:energy_continuous_continuous_damping}
Consider the Laplace transformed PML equation \eqref{eq:reduced_laplace_forcing} in a piece-wise continuous heterogeneous media with continuously varying damping coefficients $d_x, d_y,d_z \geq0$, subject to   boundary condition \eqref{eq:general_bc_laplace_pml_2}, $\lvert r\rvert \le 1$, and interface conditions \eqref{eq:interface_condition_laplace_pml_2} at discontinuities in  material parameters. Then for any $s\in\mathbb{C}$ with $\textrm{Re}(s) = a>\beta$, where $\beta =\max_{\eta = x, y, z,\Omega}\left(\sqrt{\kappa/\rho}{\lvert d^{\prime}_\eta(\eta)/d_{\eta}\rvert} \right)$,  a solution satisfies
\begin{align}\label{eq:energy_est}
(a-\beta) E_{\widehat{u}}^2(s) + \sum_{\eta = x,y,z}\left\|\frac{1}{S_{\eta}}\pdd{{\widehat{u}}}{\eta}\right\|_{k_{\eta}/\rho}^2 + BT_r(\widehat{u}) \leq 2 E_{\widehat{u}}(s)E_f(s),
\end{align}
where 

$$
BT_r(\widehat{u}) = 
  \begin{cases}
  0, \qquad r =-1\\
  \int_{\partial \Omega}\left(\frac{1-r}{1+r}c\lvert(s\widehat{u})\rvert^2\sum_{\eta = x,y,z}\frac{a(a+d_\eta)+ b^2 }{(a+d_{\eta})^2 + b^2}n_{\eta}^2\right) dS \ge 0,\quad r \ne -1,
  \end{cases}
  $$

$E_u$ and $E_f$ are defined in \eqref{eq:Eu} and \eqref{eq:Ef}, $n_\eta$ are the  components of the normal, and $b=Im(s)$.
\end{theorem}
\begin{proof}
We multiply \eqref{eq:reduced_laplace_forcing} with $(s\widehat{u})^*/(S_xS_yS_z)$ from the left and integrate over our domain $\Omega$, we have

\begin{align*}
\left(\frac{s}{\kappa}\left(s\widehat{u}\right), s\widehat{u}\right) - \sum_{\eta = x,y,z}\left(\frac{1}{ S_{\eta}}\pdd{ }{\eta}\left(\frac{1}{\rho S_{\eta}}\pdd{\widehat{u}}{\eta}\right), s\widehat{u}\right)
=
\left(\frac{1}{\kappa}\widehat{F}, s\widehat{u}\right)
+
\left(\div\widehat{\mathbf{f}}, s\widehat{u}\right)
\end{align*}
Integrating by part gives us
\begin{align*}
&\left(\frac{s}{\kappa}\left(s\widehat{u}\right), s\widehat{u}\right) 
+ 
\sum_{\eta = x,y,z}\left(\frac{(sS_{\eta})^*}{\rho S_{\eta}}\left(\frac{1}{S_{\eta}}\pdd{\widehat{u}}{\eta}\right), \frac{1}{S_{\eta}}\pdd{\widehat{u}}{\eta}\right)\\
&+
\sum_{\eta = x,y,z}\left(\pdd{}{\eta}\left(\frac{1}{S_{\eta}}\right)\left(\frac{1}{\rho S_{\eta}}\pdd{\widehat{u}}{\eta}\right),s\widehat{u}\right)\\
&-
\sum_{\eta = x,y,z}\bsprod{\frac{1}{\rho}\left(\frac{1}{S_{\eta}^2} \pdd{\widehat{u}}{\eta}n_{\eta}\right)}{s\widehat{u}}
=
\left(\frac{1}{\kappa}\widehat{F}, s\widehat{u}\right)
+
\left(\div\widehat{\mathbf{f}}, s\widehat{u}\right),
\end{align*}
where $n_{\eta}$ is the $\eta$ component of the normal. Note that the last term in the first row of the above vanishes if the damping is constant in space.
At discontinuities in  $\rho$ the interface terms vanishes due to the interface conditions \eqref{eq:rotated_IC_0}, when we take into account that the damping is continous. At the boundaries we enforce the boundary conditions \eqref{eq:rotated_bc_0}. Note that if $r = -1$ we have $(s\widehat{u}) = 0$ and the boundary integral vanishes. And if $r \ne -1$ we enforce the boundary conditions and we have

\begin{align*}
&\left(\frac{s}{\kappa}\left(s\widehat{u}\right), s\widehat{u}\right) 
+ 
\sum_{\eta = x,y,z}\left(\frac{(sS_{\eta})^*}{\rho S_{\eta}}\left(\frac{1}{S_{\eta}}\pdd{\widehat{u}}{\eta}\right), \frac{1}{S_{\eta}}\pdd{\widehat{u}}{\eta}\right)
+
\sum_{\eta = x,y,z}\left(\pdd{}{\eta}\left(\frac{1}{S_{\eta}}\right)\left(\frac{1}{\rho S_{\eta}}\pdd{\widehat{u}}{\eta}\right), s\widehat{u}\right)\\
&+
\sum_{\eta = x,y,z}\bsprod{\frac{1-r}{1+r}c\frac{n_{\eta}^2}{S_{\eta}}\left(s\widehat{u}\right)}{s\widehat{u}}
=
\left(\frac{1}{\kappa}\widehat{F}, s\widehat{u}\right)
+
\left(\div\widehat{\mathbf{f}}, s\widehat{u}\right)
\end{align*}
Note also  that the boundary integral vanishes when $r = 1$.  
Adding the conjugate transpose yields

\begin{align*}
&2(a-\beta)\left(\left(\frac{1}{\kappa}\left(s\widehat{u}\right), s\widehat{u}\right) 
+ 
\sum_{\eta = x,y,z}\left(\frac{1}{\rho }\left(\frac{1}{S_{\eta}}\pdd{\widehat{u}}{\eta}\right), \frac{1}{S_{\eta}}\pdd{\widehat{u}}{\eta}\right)\right)\\
&+ 
2\sum_{\eta = x,y,z}\left(\frac{\kappa_\eta}{\rho }\left(\frac{1}{S_{\eta}}\pdd{\widehat{u}}{\eta}\right), \frac{1}{S_{\eta}}\pdd{\widehat{u}}{\eta}\right)
+
2\sum_{\eta = x,y,z}\bsprod{\frac{1-r}{1+r}c\frac{a(a+d_\eta)+ b^2 }{(a+d_{\eta})^2 + b^2}n_{\eta}^2\left(s\widehat{u}\right)}{s\widehat{u}}\\
&\leq
\left(\frac{1}{\kappa}\widehat{F}, s\widehat{u}\right)
+
\left(s\widehat{u}, \frac{1}{\kappa}\widehat{F}\right)
-
\left(\widehat{\mathbf{f}}, s\grad\widehat{u}\right)
-
\left(s\grad\widehat{u}, \widehat{\mathbf{f}}\right)
\end{align*}
where $\beta$ satisfies
$$
\beta =\max_{\eta = x, y, z,\Omega}\left(\sqrt{\kappa/\rho}{\lvert d^{\prime}_\eta(\eta)/d_{\eta}\rvert} \right).
$$
Here we have used
$$\left\lvert\pdd{}{\eta}\left(\frac{1}{S_{\eta}}\right) \right\rvert \le {\left\lvert\frac{d^{\prime}_\eta(\eta)}{d_{\eta}}\right\rvert}
$$
Using Cauchy Schwarz inequality and
$$
\left\lvert \frac{\left(sS_\eta\right)^*}{sS_\eta}\right\rvert = 1,
$$
 give us the desired estimate, which   completes the proof.
\end{proof}
\begin{remark}
In particular, constant damping coefficients gives us $\beta = 0$ and the proof holds for any $a>0$.
\end{remark}

\begin{corollary}\label{cor:energy}
Consider the problem defined in Theorem \ref{thm:energy_continuous_continuous_damping} with  constant damping $d_x,d_y,d_z \geq 0$. The problem is asymptotically stable in the sense that no exponentially growing solutions are supported.
\end{corollary}

Transforming our estimate in Theorem \ref{thm:energy_continuous_continuous_damping} from Laplace domain into time domain is hard to do while keeping the sharpness of the bound.
However we can straight-forwardly achieve a bound that increases exponentially with time.
This does not prove asymptotic stability but it is enough to assure well-posedness of the PML problem.
Let $\mathcal{L}^{-1}$ denote the inverse Laplace transform and introduce the following norms
\begin{equation}\label{eq:Eu_inv}
\begin{split}
E_u^2(t,d_{\eta}) &= \left\|\mathcal{L}^{-1}(s\widetilde{u})\right\|^2_{1/\kappa} + \sum_{\eta = x,y,z}\left\|\mathcal{L}^{-1}\left(\frac{1}{S_{\eta}}\pdd{\widetilde{u}}{\eta}\right)\right\|^2_{1/\rho}\\
&=\left\|\pdd{u}{t}\right\|^2_{1/\kappa} +  \sum_{\eta = x,y,z}\left\|\pdd{u}{\eta}-d_{\eta}e^{-d_{\eta}t}\pdd{u}{\eta}\right\|^2_{1/\rho}>0,
\end{split}
\end{equation}
and
\begin{equation}\label{eq:Ef_inv}
\begin{split}
E_f^2(t, d_\eta)  = \left\|\mathcal{L}^{-1}\left(\widehat{F}\right)\right\|_{1/\kappa}^2+\left\|\mathcal{L}^{-1}\left(\mathbf{\widehat{f}}\right)\right\|_{\rho}^2 >0.
\end{split}
\end{equation}

\begin{theorem}\label{thm:energy_continuous_2}
Consider the Laplace transformed PML equation \eqref{eq:reduced_laplace} with  constant damping $d_{\eta}\geq 0$ and subject to homogeneous initial data. Let the energy norms $E_u^2(t,d_{\eta})>0$, $E_f^2(t,d_{\eta})>0$ be defined as \eqref{eq:Eu_inv} and \eqref{eq:Ef_inv}. Then
\begin{equation}\label{eq:time_estimate}
\int_0^Te^{-at}E_u^2(t,d_{\eta})dt\leq \frac{4}{(a-\beta)^2}\int_0^Te^{-2at}E_f^2(t,d_{\eta})dt
\end{equation}
for any $a>\beta $ and $ T>0.$

\end{theorem}
\begin{proof}
The proof can be found in \cite{duru2019energy}.
\end{proof}

From the results in this section we can conclude that our PML problem is well posed. In particular, for constant damping \eqref{eq:time_estimate} is valid for any $a>0$, and we can choose $a>0$ in relation to the time interval. A bound in physical space is obtained, which grows only algebraically in time. With variable damping the bound grows exponentially.
In section \ref{sec:discrete} we will see that we can follow the same approach to prove stability in the corresponding sense also for the discrete problem.

\section{The finite element approximation\label{sec:fem}}
In this section we formulate the PML problem in weak form, and state the corresponding FEM. 
For the weak form we will consider $u, \psi \in V$, where
$$
V =  
 \begin{cases}
H_0^1(\Omega),\quad r =-1,\\
H^1(\Omega),\quad r \ne -1,
\end{cases}
$$
and $\boldsymbol{\phi} \in [L^2(\Omega)]^3$.  Note that $u\in V$ is the standard choice for the wave equation without damping. In the PML $u\in V$ in \eqref{eq:pml_system_3D} implies $\boldsymbol{\phi} \in [L^2(\Omega)]^3$ and $\psi \in V$. Here $H_0^1(\Omega)$ and $H^1(\Omega)$ are the standard Sobolev spaces. 
We begin by multiplying the PML system  in \eqref{eq:pml_system_3D} by test functions $v,q \in V $ and $\mathbf{p}\in [L^2(\Omega)]^3$ and integrating over the domain.
We obtain
\begin{subequations}
\label{eq:pml_system_3D_strong}
\begin{alignat}{3}
&\left(\frac{1}{\kappa}\ddot{u} + \frac{1}{\kappa}\sum_{\eta} d_\eta\dot{u} + \frac{d_xd_yd_z}{\kappa}\sum_{\eta}\frac{1}{d_\eta} u, v\right) \nonumber  \\
&=  
\left(\div\left(\frac{1}{\rho} \grad {u}\right) + \div{\boldsymbol{\phi}}  -\frac{d_xd_yd_z}{\kappa} \psi + \frac{1}{\kappa}F_{u}, v\right),\label{eq:pml_1_3D} \\ 
  &\left(\dot{\boldsymbol{\phi}}, \mathbf{p}\right) = \left(-\mathbf{d} \boldsymbol{\phi}+\frac{1}{\rho}\boldsymbol{\Gamma} \grad{u}+\frac{1}{\rho} \boldsymbol{\Upsilon} \grad{\psi} + \mathbf{f}_{\phi}, \mathbf{p}\right), 
  \label{eq:pml_2_3D}
  \\
		&\left(\dot{\psi}, q\right) = \left(u + f_{\psi}, q\right).\label{eq:pml_5_3D}
\end{alignat}
\end{subequations}
When $r=-1$ we have a Dirichlet boundary condition, which is imposed strongly in the function space for $u$ and $v$.
When $r\neq-1$ we need to consider the boundary condition multiplied by the test function $v$ and integrated over the boundary,
\begin{equation}\label{eq: weak_boundary}
\frac{1-r}{2}\bsprod{c\left(\dot{u}+\boldsymbol{\theta}\cdot \mathbf{n}\right)}{v}+\frac{1+r}{2}\bsprod{\frac{1}{\rho}\grad u\cdot \mathbf{n}+\boldsymbol{\phi}\cdot \mathbf{n}}{v}=0.
\end{equation}
Similarly, we will use the interface condition multiplied by the test function $v$, and integrated over the interface,
\begin{equation}\label{eq:weak_interface}
\bsprod{\lJump\frac{1}{\rho}\grad{u}\cdot \mathbf{n}\rJump  + {\lJump\boldsymbol{\phi} \cdot \mathbf{n}\rJump} }{ v}=0.
\end{equation}
Next we perform integration by parts on the right hand side of equation \eqref{eq:pml_1_3D}. If we take (\ref{eq:weak_interface}) into account the result is
%
\begin{align}
&\left(\frac{1}{\kappa}\ddot{u} + \frac{1}{\kappa}\sum_{\eta} d_\eta\dot{u} + \frac{d_xd_yd_z}{\kappa}\sum_{\eta}\frac{1}{d_\eta} u, v\right)  
=  
-\left(\frac{1}{\rho} \grad {u}, \grad{v}\right) - \left({\boldsymbol{\phi}}, \grad{v}\right)
\nonumber
\\
&-\left(\frac{d_xd_yd_z}{\kappa} \psi, v\right) + \left(\frac{1}{\kappa}F_{u}, v\right) 
+ \bsprod{\frac{1}{\rho}\grad{u}\cdot \mathbf{n}+ \boldsymbol{\phi} \cdot \mathbf{n}}{ v},  \label{eq:pml_1_3D_0} \\ 
  &\left(\dot{\boldsymbol{\phi}}, \mathbf{p}\right) = \left(-\mathbf{d} \boldsymbol{\phi}+\frac{1}{\rho}\boldsymbol{\Gamma} \grad{u}+\frac{1}{\rho} \boldsymbol{\Upsilon} \grad{\psi} + \mathbf{f}_{\phi}, \mathbf{p}\right), 
  \label{eq:pml_2_3D_0}
  \\
  &\left(\dot{\psi}, q\right) = \left(u + f_{\psi}, q\right).\label{eq:pml_3_3D_0}
\end{align}
When $r\neq-1$
we use (\ref{eq: weak_boundary}) to express the boundary term in $\dot{u}$ and $\theta$, and thus the boundary term can be written as
$$
BT_r(u,\boldsymbol{\theta},v):= 
  \begin{cases}
0, \qquad r =-1,\\
  \bsprod{c\frac{(1-r)}{r+1} \left(\dot{u} + \boldsymbol{\theta} \cdot \mathbf{n}\right)}{v},\quad r \ne -1,
  \end{cases}
$$

When $r=-1$ the term vanishes due to the strong Dirichlet condition imposed on the function space.

\begin{remark}
Note that both the first and second term of the right hand side of \eqref{eq:pml_1_3D} are integrated by parts.
This weak form will allow the discrete energy analysis in the next section to follow the steps in Section \ref{sec:laplace}. We also note that with this formulation the interface condition is imposed weakly.
\end{remark}
\begin{remark}
From experience we know that naively only integrating the first term of the right hand side of \eqref{eq:pml_1_3D} by parts, yields an unstable numerical method.
\end{remark}
The weak PML problem can now be stated.
For all $v, q \in V$ and $\boldsymbol{p} \in [L^2\left(\Omega\right)]^3$,  find $u, \psi \in V$ and $\boldsymbol{\phi} \in [L^2\left(\Omega\right)]^3$  such that
\begin{subequations}
\label{eq:pml_system_3D_weak_form}
\begin{alignat}{5}
&\left(\frac{1}{\kappa}\ddot{u} + \frac{1}{\kappa}\sum_{\eta} d_\eta\dot{u} + \frac{d_xd_yd_z}{\kappa}\sum_{\eta}\frac{1}{d_\eta} u, v\right)  
= -\left(\frac{1}{\rho} \grad {u}, \grad{v}\right)\nonumber \\ &- \left({\boldsymbol{\phi}}, \grad{v}\right)-\left(\frac{d_xd_yd_z}{\kappa} \psi, v\right) \label{eq:pml_1_3D_weak}- BT_r(u,\boldsymbol{\theta},v)+ \left(\frac{1}{\kappa}F_{u}, v\right)  \\ 
  &\left(\dot{\boldsymbol{\phi}}, \mathbf{p}\right) = \left(-\mathbf{d} \boldsymbol{\phi}+\frac{1}{\rho}\boldsymbol{\Gamma} \grad{u}+\frac{1}{\rho} \boldsymbol{\Upsilon} \grad{\psi} + \mathbf{f}_{\phi}, \mathbf{p}\right), 
  \label{eq:pml_2_3D_weak}
  \\
		&\left(\dot{\psi}, q\right) = \left(u + f_{\psi}, q\right).\label{eq:pml_3_3D_weak}
\end{alignat}
\end{subequations}
Note that for $|r| = 1$ the boundary term vanishes identically. Since we derived this weak form from the strong form consistency follows.

It remains to formulate the  Galerkin approximation. 

For this purpose we discretise our domain $\Omega$ by introducing a tesselation $K_{\Omega}$, with no hanging nodes, consisting of elements, $K$, that completely cover $\Omega$. More precisely
\begin{equation}
\Omega =\bigcup_{K\in K_{\Omega}} K.
\label{eq:mesh}
\end{equation}
Next we introduce our discrete spaces
\begin{equation}
V_h={\{v\in V : v\lvert_K \in N_p(K), \, K \in K_{\Omega} \}}
\label{eq:fem_V_space}
\end{equation}
and 
\begin{equation}\label{eq:fem_V_space}
W_h=\{ \mathbf{v}\in [L^2\left(\Omega\right)]^3 : \mathbf{v}\lvert_K \in N_p(K), \, K \in K_{\Omega} \}.
\end{equation}
where $N_p(K)$ denotes an element with polynomial order $p \geq 1$, over $K$.  

The Galerkin approximation for the PML problem can now be stated.
For all $v, q \in V_h $ and $\boldsymbol{p} \in W_h$,  find $u_h, \psi_h \in V_h$ and $\boldsymbol{\phi}_h \in W_h$  such that
\begin{subequations}
\label{eq:pml_system_3D_weak_form_disc}
\begin{alignat}{4}
\nonumber
&\left(\frac{1}{\kappa}\ddot{u}_h + \frac{1}{\kappa}\sum_{\eta} d_\eta\dot{u}_h + \frac{d_xd_yd_z}{\kappa}\sum_{\eta}\frac{1}{d_\eta} u_h, v\right)  
= -\left(\frac{1}{\rho} \grad {u}_h, \grad{v}\right)\\ &- \left({\boldsymbol{\phi}}_h, \grad{v}\right)-\left(\frac{d_xd_yd_z}{\kappa} \psi_h, v\right) \label{eq:pml_1_3D_weak_disc}
- BT_r(u_h,\boldsymbol{\theta}_h,v)+ \left(\frac{1}{\kappa}F_{u}, v\right)  \\ 
  &\left(\dot{\boldsymbol{\phi}}_h, \mathbf{p}\right) = \left(-\mathbf{d} \boldsymbol{\phi}_h+\frac{1}{\rho}\boldsymbol{\Gamma} \grad{u}_h+\frac{1}{\rho} \boldsymbol{\Upsilon} \grad{\psi}_h + \mathbf{f}_{\phi}, \mathbf{p}\right), 
  \label{eq:pml_2_3D_weak_disc}
  \\
		&\left(\dot{\psi}_h, q\right) = \left(u_h + f_{\psi}, q\right).\label{eq:pml_3_3D_weak_disc}
\end{alignat}
\end{subequations}

The initial data \eqref{eq:u0} is $L^2$-projected onto the finite element space.

The numerical experiments will be performed in 2D. There the approximation reads as in 3D, but the number of auxiliary variables is reduced to two, corresponding to the first two components of $\phi$ in the 3D case.  
\section{Discrete Analysis\label{sec:discrete}}
 In this section we derive energy and convergence results for particular cases. The energy result is quite general and holds for continuously varying damping coefficient, while the convergence proof is more restrictive and relies on constant damping coefficients. For simplicity we will consider Dirichlet boundary condition, but the analysis can easily be extended to the more general boundary conditions.

\subsection{Discrete energy analysis}
Here, we derive an energy estimate for the discrete problem with a particular choice of quadrature, which corresponds to the polynomial order.  In the final part of the derivation we will assume that in each element the quadrature points can be used to interpolate any function to the local polynomial space.  For later use we introduce 
\begin{definition}
The quadrature and the polynomial space satisfy the Interpolation Condition if any function values at the quadrature points in each element can be interpolated by a polynomial, which belongs to the local polynomial space. For each element $T$ the local interpolating polynomial defines the  Quadrature Point Interpolant $I_q(T)$. By $I_q$ we denote the corresponding global Quadrature Point Interpolant.
\end{definition}

For example, with rectangular elements we can use a tensor product Lagrange basis  with piece-wise polynomials of degree $p$ in each spatial direction based on  Lagrange-Gauss-Lobatto nodes, and a corresponding Gauss-Legandre quadrature with $p+1$ quadrature points in each direction, to satisfy the Interpolation Condition. Such a quadrature rule can integrate polynomials of order $2p+1$ exactly. If a higher order quadrature involving more points is used, the Interpolation Condition will not be satisfied.
Comments on how to extend the result to the other quadratures or exact integration will follow in the end of the section. 

We will analyze the discretization in Laplace space, and therefore we  consider the Laplace transform of the discrete weak form,
\begin{subequations}
\label{eq:pml_system_3D_weak_form_laplace}
\begin{alignat}{4}
&\left(\frac{s^2}{\kappa}\hat{u}_h + \frac{s}{\kappa}\sum_{\eta} d_\eta \hat{u}_h + \frac{d_xd_yd_z}{\kappa}\sum_{\eta}\frac{1}{d_\eta} \hat{u}_h, v\right)_{h} 
=  
-\left(\frac{1}{\rho} \grad {\hat{u}_h}, \grad{v}\right)_h \nonumber\\ & - \left({\boldsymbol{\hat{\phi}_h}}, \grad{v}\right)_h -\left(\frac{d_xd_yd_z}{\kappa} \hat{\psi}_h, v\right)_h + \left(\frac{1}{\kappa}\hat{F}_{u}, v\right)_{h} \label{eq:pml_1_3D_weak_form_laplace} \\ 
  &\left(s\hat{\boldsymbol{\phi}}_h, \mathbf{p}\right)_h = \left(-\mathbf{d} \hat{\boldsymbol{\phi}}_h+\frac{1}{\rho}\boldsymbol{\Gamma} \grad{\hat{u}_h}+\frac{1}{\rho} \boldsymbol{\Upsilon} \grad{\hat{\psi}_h}, \mathbf{p}\right)_h, 
  \label{eq:pml_2_3D_weak_form_laplace}
  \\
		&\left(s\hat{\psi}_h, q\right)_h = \left(\hat{u}_h,q\right)_h, \label{eq:pml_3_3D_weak_form_laplace}
\end{alignat}
\end{subequations}
The effects of initial data is incorporated in the forcing function $\hat{F}_{u_h}$, but here we have for brevity assumed forcing only in the main equation. Here $v\in V_h$ and $\boldsymbol{w}\in W_h$. Further, the discrete inner product used in (\ref{eq:pml_1_3D_weak_form_laplace})-(\ref{eq:pml_3_3D_weak_form_laplace}), and the corresponding norm are defined by quadrature instead of integrals.

The next step is to derive a discrete reduced weak form involving only $\hat{u}_h$ and $v$. We would like to use  $\boldsymbol{\tilde p}(\mathbf{x})=\left(s^{*}I+\boldsymbol{d}(\mathbf{x})\right)^{-1}\grad v$ as test function in (\ref{eq:pml_2_3D_weak_form_laplace}). Here $I$ is the identity matrix and $\boldsymbol{d}$ is a diagonal matrix defined in section 3. However, for variable damping $\boldsymbol{\tilde p}$ is not in $W_h$ even if $v\in V_h$. If the Interpolation Condition (see definition 1) is satisfied, the corresponding  Quadrature Point Interpolant can be used, and we have
\begin{theorem}\label{thm:reducedNew}
Let $\hat{u}_h$, $\boldsymbol{\hat{\phi}_h}$, and $\hat{\psi}_h$ satisfy the discrete weak form in Laplace space, (\ref{eq:pml_1_3D_weak_form_laplace})-(\ref{eq:pml_3_3D_weak_form_laplace}). For constant damping functions the solution $\hat{u}_h$  satisfies the reduced weak form
\begin{equation}\label{eq:weakform_reduced2}
\left(\frac{s^2S_xS_yS_z}{\kappa}\hat{u}_h,  {v}\right)_h = -\left(\frac{1}{\rho}\boldsymbol{S}\grad_d\hat{u}_h,\grad v\right)_h + \left(\frac{1}{\kappa}\hat{F}_{u}, v\right)_h, \quad \forall v\in V_h. 
\end{equation}
If the Interpolation Condition (see definition 1) is fulfilled, then the solution $\hat{u}_h$ satisfies \eqref{eq:weakform_reduced2} also in the case of spatially varying damping functions. 
\end{theorem}
\begin{proof}
Consider the case of spatially varying damping functions. The constant case follows in the same way. In \eqref{eq:pml_3_3D_weak_form_laplace} we note that $\hat{\psi}_h$ and $\hat{u}_h$ are in the same space, which is spanned by the test functions $q$. Thus  we have 
$$
\hat{\psi}_h = \frac{1}{s}\hat{u}_h.
$$
We can thus eliminate $\hat{\psi}_h$ from \eqref{eq:pml_2_3D_weak_form_laplace} and \eqref{eq:pml_1_3D_weak_form_laplace}. 
Next we let 
$$\boldsymbol{p_h}(\mathbf{x})=I_q(T_j)(s^*I+\boldsymbol{d}(\mathbf{ x}))^{-1}\grad v(\mathbf{x}),\quad\mathbf{x}\in T_j,$$
for any $v\in V_h$. By the Interpolation Condition $\boldsymbol{p_h}\in W_h$. With $\boldsymbol{p_h}$ as test function in \eqref{eq:pml_2_3D_weak_form_laplace} we get
\begin{equation}
\nonumber
\left({\boldsymbol{\hat{\phi}_h}},\grad v\right)_h=
\left(\frac{1}{\rho}\left(\Gamma+\frac{1}{s}\Upsilon\right)\grad \hat{u}_h, (s^*I+\boldsymbol{d})^{-1}\grad v\right)_h.
\end{equation}
By the definitions of the diagonal matrices $\Gamma$ and $\Upsilon$, this relation can be rewritten as
\begin{equation}\nonumber
\left({\boldsymbol{\hat{\phi}_h}},\grad v\right)_h=
\left(\frac{1}{\rho}\boldsymbol{S}\grad_d\hat{u}_h,\grad v\right)_h-\left(\frac{1}{\rho}\grad\hat{u}_h,\grad v\right)_h,
\end{equation}
and we can also  eliminate $\boldsymbol{\hat{\phi}_h}$ from \eqref{eq:pml_1_3D_weak_form_laplace}, arriving at the problem in reduced weak-form, \eqref{eq:weakform_reduced2}.
\end{proof}

If the Interpolation Condition is not satisfied we need to find a suitable test function $\boldsymbol{\tilde{p}_h}\in W_h$ that will allow eliminating $\boldsymbol{\hat{\phi}_h}$ from \eqref{eq:pml_1_3D_weak_form_laplace}.
We will use a projection of $\boldsymbol{\tilde p}(\mathbf{x})=\left(s^{*}I+\boldsymbol{d}(\mathbf{x})\right)^{-1}\grad v$ defined by
\begin{definition}
 For $\boldsymbol{g}\in W_h$ the projection $\boldsymbol{ g_p}=\Pi_p (s^{*}I+\mathbf{d})^{-1}\boldsymbol{g}\in W_h$ is defined as the solution to
\begin{equation}
 \left(\boldsymbol{w_h},(s^{*}I+\boldsymbol{d})\boldsymbol{g_p} \right)_h=
\left(\boldsymbol{w_h},\boldsymbol{g}\right)_h, \quad \forall \boldsymbol{w_h}\in W_h.
\end{equation}
\end{definition}
Now introduce $\boldsymbol{ p_h}=\Pi_p(s^{*}I+\boldsymbol{d})^{-1}\grad v$ as test function in \eqref{eq:pml_2_3D_weak_form_laplace}, and get
\begin{equation}
\nonumber
\left({\boldsymbol{\hat{\phi}_h}},\grad v\right)_h=
\left(\frac{1}{\rho}\left(\Gamma+\frac{1}{s}\Upsilon\right)\grad \hat{u}_h,\Pi_p (s^*I+d)^{-1}\grad v\right)_h.
\end{equation}
By the definitions of the diagonal matrices $\Gamma$ and $\Upsilon$, this relation can be rewritten similarly as above,
\begin{equation}\label{eq:reduce}
\left(\boldsymbol{\hat{\phi}_h},\grad v\right)_h=
\left(\frac{1}{\rho}\boldsymbol{S}\grad_d \hat{u}_h,(s^*I+d)\Pi_p(s^*I+d)^{-1}\grad v\right)_h-\left(\frac{1}{\rho}\grad\hat{u}_h,\grad v\right)_h,
\end{equation}
and we can again  eliminate $\boldsymbol{\hat{\phi}_h}$ from \eqref{eq:pml_1_3D_0}, arriving at the problem in reduced weak-form. 
This proves the more general theorem
\begin{theorem}\label{thm:reduced}
Let $\hat{u}_h$, $\hat{\phi}_h$, and $\hat{\psi}_h$ satisfy the discrete weak form in Laplace space (\ref{eq:pml_1_3D_weak_form_laplace})-(\ref{eq:pml_3_3D_weak_form_laplace}). Then $\hat{u}_h$ also satisfies the reduced weak form
\begin{equation}\label{eq:weakform_reduced}
\left(\frac{s^2 S_xS_yS_z}{\kappa}\hat{u}_h,  {v}\right)_h
=  -\left(\frac{1}{\rho}\boldsymbol{S}\grad_d\hat{u}_h,(s^*I+d)\Pi_p(s^*I+d)^{-1}\grad v\right)_h
 + \left(\frac{1}{\kappa}\hat{F}_{u}, v\right)_h, \forall v\in V_h.
\end{equation}
Here $\Pi_p$ is defined in Definition 2.
\end{theorem}

For simplicity we will in the remainder of this section assume that the Interpolation Condition is satisfied. It follows that the reduced weak form is \eqref{eq:weakform_reduced2}.

We will use a similar energy for the approximate solution in Laplace space, $\hat{u}_h$, as for the exact solution
\begin{align}\label{eq:Eu_h}
E^2_{\hat{u}_h}(s) &= \left(\frac{1}{\kappa}s\hat{u}_h,s\hat{u}_h\right)_{h} + \sum_{\eta = x,y,z}\left(\frac{1}{S_{\eta}\rho} \pdd{\hat{u}_h}{\eta},\frac{1}{S_{\eta}}\pdd{\hat{u}_h}{\eta}\right)_{h} > 0
\end{align}
and an energy for the data
\begin{align}\label{eq:Ef_h}
E_{f_h}^2(s) &= \left(\frac{1}{\kappa}\hat{F}_u,\hat{F}_u\right)_{h} > 0
\end{align}

The discrete norms are defined using the same quadrature points as above.
We can now formulate a stability result.

\begin{theorem}\label{thm:energy_discrete}
Consider the discrete weak form \eqref{eq:weakform_reduced2}
in a piece-wise continuous heterogeneous media and with continuous damping $d_{\eta} (\eta)\geq0$, with homogeneous initial data. Assume that the Interpolation Condition (see Definition 1) is satisfied.
There exists a $\beta_h=\beta+{\mathcal O}(h^{p-1})$, with $p$ being the polynomial order of the discrete space, such that for any $\textrm{Re}(s) = a>\beta_h$  the solution $\hat{u}_h\in V_h$ satisfies 
\begin{align}\label{eq:energy_est_discrete}
(a -\beta_h)E_{\hat{u}_h}^2(s) 
\leq 2 E_{\hat{u}_h}(s)E_{f_h}(s).
\end{align}
Here $E_{\hat{u}_h}$, $E_{f_h}$ are defined in \eqref{eq:Eu_h}--\eqref{eq:Ef_h} and $\beta$ is defined in Theorem \ref{thm:energy_continuous_continuous_damping}.
\end{theorem}

\begin{proof}
 Let $v=sI_q\left(S_x^*S_y^*S_z^*\right)^{-1}\hat{u}_h$, where $I_q$ is the Quadrature Point Interpolant defined in Definition 1. By assumption $v\in V_h$. Note that in the quadrature points $v(\bar q_i^{(j)})=s\left(S_x^*S_y^*S_z^*\right)^{-1}\hat{u}_h$ at ${\bar{q}_i^{(j)}}$, and that for variable damping parameters the interpolation is needed to ensure that $v$ is in the test space. 
It follows that
\begin{equation}\label{eq:weak_laplace_reduced_00}
\begin{split}
\left(\frac{s^2}{\kappa}\hat{u_h},(s\hat{u_h})\right)_{h}  +
\sum_{\eta = x, y, z}\left(\frac{1}{\rho}\frac{1}{S_\eta^2} \pdd{\hat{u_h}}{\eta},\pdd{(s\hat{u_h})}{\eta}\right)_{h}+\\
+\sum_{\eta = x, y, z}\left(\frac{1}{\rho}\frac{1}{S_\eta} \pdd{\hat{u_h}}{\eta},\pdd{}{\eta}\left(I_p\frac{1}{S_\eta}\right)s\hat{u_h}\right)_{h}
=\left(\frac{1}{\kappa}\hat{F}_u, (s\hat{u_h})\right)_h \\
\end{split}
\end{equation}
Note that
$$
\pdd{}{\eta}\left(I_p\frac{1}{S_\eta}\right)=\pdd{}{\eta}\left(\frac{1}{S_\eta}\right)+{\mathcal O}(h^{p-1}).
$$
Thus adding the conjugate transpose of \eqref{eq:weak_laplace_reduced_00} and using Cauchy-Schwarz inequality with
\begin{equation*}
\bigl\lvert \frac{(sS_\eta)^*}{sS_\eta}\bigr\rvert = 1,
\end{equation*}
 gives us the desired estimate, which completes the proof.
\end{proof}

\begin{remark}
There is a corresponding stability estimate for the discrete solution in physical space. As for the continuous problem,  the estimate excludes exponential growth when the damping is constant, while with variable damping it allows  bounded, exponential growth.
\end{remark}
\begin{remark}  
If some other quadrature is used the more general reduced form, \eqref{eq:weakform_reduced}, is valid. Note that the difference between integral values when using different quadratures  is proportional to $h^p$ for some $p>0$. Therefore we expect that when using other quadratures, or exact integration, the discrete solution $\hat{u}_h$ will satisfy a similar, but slightly perturbed energy estimate. 
\end{remark}

\subsection{Convergence}

Next, we turn our attention to the question of convergence.   We will prove a Theorem which is valid for constant damping coefficients, $d_{\eta} \geq 0$. A discussion on how to extend the convergence proof to include continuously varying damping is given at the end of this subsection. However, the lemmas in this subsection are valid also for continuously varying damping functions. In this subsection we will use the Sobolev norm $\|\cdot \|_{H^{k}} = \|\cdot \|_{W^{k}_2}$ as well as the corresponding semi-norm $\lvert\cdot \rvert_{H^{k}} = \lvert\cdot \rvert_{W^{k}_2}$.

 Let $I_h:H^1(\Omega)\rightarrow V_h$, where $V_h$ is the space of piece-wise polynomials of degree $k$, be the standard Cl{\'e}ment interpolation operator, see for example \cite{ern2004theory}.  We will use the following variant of the standard interpolation result 
\begin{equation}\label{eq:interp}
\|\boldsymbol{\nabla}_d(u - I_hu)\| \leq \|\grad(u - I_hu)\| \leq C_I h^{l} \lvert u \rvert_{H^{l+1}(\Omega)},
\end{equation}
where $C_I$ is a constant independent of $u,\,h$ and $s$, and $1\leq l\leq k$. Note that the first inequality in \eqref{eq:interp} holds since $\lvert \frac{1}{S_{\eta}}\rvert\leq 1$. The second inequality holds as long as $u\in H^{l+1}(\Omega)$. Next we introduce a bi-linear form that will be useful in our convergence proof.
Let $u$ and $v$ be complex valued functions and define 
\begin{equation}\label{eq:aux_vf}
A(u,v) = \sum_{\eta = x,y,z} (s^*\frac{1}{\rho}\frac{1}{S_{\eta}^2} \pdd{u}{\eta}, \pdd{v}{\eta}).
\end{equation}
\begin{lemma}\label{lem:bilinear_A}
Consider $Re(s)=a>0$ and a given function $f\in L_2$. The bilinear form \eqref{eq:aux_vf} can be used to formulate two problems. Find $u,w \in V$ s.t. 
\begin{equation}\label{eq:aux_vf_adjoint}
A(u,v) =(f,v),\, \forall v \in V,\quad A(v,w)=(v,f), \, \forall v \in V,
\end{equation}
respectively. 
The following inequality holds
\begin{equation}\label{eq:A_bounded}
\lvert A(u,v)\rvert\leq \|\boldsymbol{\nabla}_d u\|_{\frac{1}{\rho}}\|\boldsymbol{\nabla}_d sv\|_{\frac{1}{\rho}},
\end{equation}
as well as the following elliptic regularity results
\begin{equation}\label{eq:elliptic_regularity}
    \|u\|_{H^2} \leq C_R\frac{1}{a} \|f\|, \quad \|w\|_{H^2} \leq C_R\frac{1}{a} \|f\|.
\end{equation}
Here $C_R$ is a constant independent of $u,\,h$ and $s$.
\end{lemma}
\begin{proof}
Note that
$Re(s^*\frac{1}{\rho}\frac{1}{S_{\eta}^2})=\frac{1}{\rho \lvert S_\eta\rvert^2}Re(\frac {s^*S_\eta^*}{S_\eta})$. By \eqref{eq:real} this quantity is positive as long as $Re(s)>0$, and ellipticity follows. Along with ellipticity we get the elliptic regularity result \eqref{eq:elliptic_regularity}, see for example \cite{BrennerScott}. The inequality \eqref{eq:A_bounded} follows easily since $\lvert S_\eta^*\rvert /\lvert S_\eta\rvert=1$.


\end{proof}
Thanks to \eqref{eq:real} the following Lemma can straightforwardly be proven.
\begin{lemma}\label{lem:bilinear_Atilde}
Consider $Re(s) = a >0$. Let $u$ and $v$ be complex valued functions, and let $A$ be the bilinear form defined in Lemma \ref{lem:bilinear_A}. For the bilinear form 
$ \tilde{A}(u,v) = A(u,v)+A^*(u,v).$
we have the following inequality 
\begin{equation}\label{eq:A_coercive}
\tilde{A}(u,u) \geq a\|\boldsymbol{\nabla}_d u\|^2_{\frac{1}{\rho}},
\end{equation}
\end{lemma}

 We can define finite element approximations corresponding to the two elliptic problems defined in Lemma 1. In particular we have
\begin{lemma} \label{lem:ritz}
Let $ u_h\in V_h\in V$ be the standard finite element approximation corresponding to the first problem defined in \eqref{eq:aux_vf_adjoint}, with $A$ being the bi-linear form \eqref{eq:aux_vf}. Then ${u}_h$ satisfies the orthogonality
\begin{equation}\label{eq:orthogonality}
A({u}_h-{u},v) = 0, \quad v \in V_h
\end{equation}
\end{lemma}
\begin{proof}
Subtracting the variational formulation from its corresponding finite element formulation gives us the result. 
\end{proof}
We can then define our Ritz-projection $R_h : V\rightarrow V_h $ as: 
\begin{equation}\label{eq:ritz-like}
A(R_h{u}-{u},v) = 0, \quad \forall v \in V_h
\end{equation}
The Ritz projection satisfies the following best approximation result
\begin{lemma}\label{lem:bar} Assume $u\in V$.
For all $Re(s)=a>0$  the Ritz projection, $R_h u$, defined in  \ref{eq:ritz-like}, satisfies the following best approximation result
\begin{equation} \label{eq:bar}
\|\boldsymbol{\nabla}_d(R_h{u}-{u}) \|_{\frac{1}{\rho}} \leq \frac{1}{a}\|\boldsymbol{\nabla}_d (sv-s u)\|_{\frac{1}{\rho}}, \quad \forall v \in V_h.
\end{equation}

\end{lemma}

\begin{proof}
For any $v\in V_h$ let $R_h{u}-{u} =R_h{u}-v+v-{u}$ then it follows
\begin{equation*}
a\|\nabla_d(R_h{u}-{u})\|^2_{\frac{1}{\rho}} \leq \tilde{A}(R_h{u}-{u},v-{u}) \leq\|\boldsymbol{\nabla}_d(R_h{u}-{u}) \|_{\frac{1}{\rho}} \|\boldsymbol{\nabla}_d (sv-s\hat{u})\|_{\frac{1}{\rho}},
\end{equation*}
where we used the orthogonality result, the Ritz projection and the inequalities \eqref{eq:A_bounded} and \eqref{eq:A_coercive}. This concludes the proof.
\end{proof}

In the next Lemma  we state the convergence of the Ritz projection.
\begin{lemma}\label{lem:ritz-conv}
Assume $Re(s) = a >0$. For all ${u} \in H^{k+1}(\Omega)$ with $k\ge 1$, $R_h {u} $ in \eqref{eq:ritz-like} satisfies the following two estimates
\begin{equation}\label{eq:ritz_error_1}
\|\boldsymbol{\nabla}_d( R_h{u}-{u})\| \leq C_1 \frac{1}{a} h^k \|s{u}\|_{\frac{1}{\rho},H^{k+1}(\Omega)},
\end{equation}
and
\begin{equation}\label{eq:ritz_error_2}
\| R_h{u}-{u}\| \leq C_2\frac{1 }{a^2} h^{k+1} \|s^2{u}\|_{\frac{1}{\rho},H^{k+1}(\Omega)}.
\end{equation}
Here $C_1$ and $C_2$ are constants independent of $u,\,h$ and $s$.
\end{lemma}

\begin{proof}
To begin with we note that the error estimate \eqref{eq:ritz_error_1} follows from Lemma \ref{lem:bar} and the interpolation result \eqref{eq:interp}.

Next we prove the estimate \eqref{eq:ritz_error_2}. Let $e = R_h{u} - {u}$ and consider the elliptic problem defined in Lemma \ref{lem:bilinear_A}. Find $w\in H^2(\Omega)$ such that
\begin{equation}\label{eq:eta-dual}
A(v,w) = (v,e), \quad \forall v\in V
\end{equation}
Note that the forcing is chosen to be the error $e$. Choosing $v = e$ and applying the definition of our Ritz projection yields
\begin{alignat*}{2}
    \|e\|^2 &= A(e,w)
    = A(e,w-I_hw) \leq\|\nabla_d se\|_{\frac{1}{\rho}} \| \nabla_d(w-I_hw) \|_{\frac{1}{\rho}}\\
    &\leq C^*_{I}\frac{1}{a}h^{k}\|s^2{u}\|_{\frac{1}{\rho},H^{k+1}}h\lvert w\rvert_{\frac{1}{\rho},H^2}
    \leq C_2\frac{1}{a^2}h^{k+1}\|s^2{u}\|_{\frac{1}{\rho},H^{k+1}}\|e\|,
\end{alignat*}
where we used the interpolation result \eqref{eq:interp} together with result \eqref{eq:ritz_error_1} and elliptic regularity \eqref{eq:elliptic_regularity}. The constants $C^*_{I}$ and $C_2 = C^*_IC_R$ are constants independent of $u,\,s$ and $h$. $C^*_I$ is obtained from applying the interpolation result \eqref{eq:interp} two times and $C_2 = C^*_IC_R$, where $C_R$ comes from applying \eqref{eq:elliptic_regularity}. This concludes the proof.
\end{proof}

Finally we will show convergence of the reduced weak form \eqref{eq:weakform_reduced2}. 
Our convergence result is for the case  of constant damping. Thus the reduced  weak form \eqref{eq:weakform_reduced2} is equivalent to
\begin{equation}\label{eq:FEM_convergence}
(\frac{s^2}{\kappa} \hat{u}_h, v) + \sum_{\eta = x,y,z}(\frac{S_{\eta}^*}{S_{\eta}}\frac{1}{\rho}\frac{1}{S_{\eta}} \pdd{\hat{u}_h}{\eta},\frac{1}{S_{\eta}} \pdd{v}{\eta}) = (F,v), \quad \forall v \in V_h
\end{equation}
Clearly the solution $\hat{u}$ of the continuous problem satisfies the same reduced weak form for $\forall v\in V$.
For this weak form we have
\begin{theorem}\label{thm:convergence}
For any $s \in \mathbf{C}$, with $Re(s) = a > 0$  the solution $ \hat{u}_h$ to \eqref{eq:FEM_convergence} with constant damping satisfies the a priori estimate
\begin{equation}
    \|\hat{u}_h-\hat{u}\| \leq C h^{k+1}\left(\frac{1}{a^3}\|s^3 \hat{u} \|_{H^{k+1}(\Omega)}+\frac{1}{a^2}\|s^2 \hat{u} \|_{H^{k+1}(\Omega)}\right).
\end{equation}
Here $\hat{u}$ is assumed to be a sufficiently smooth solution of the corresponding continuous problem and $C$ is a constant independent of $h,\,\hat u$ and $s$.
\end{theorem}
\begin{proof}
Let $e = \hat{u}_h - \hat{u}$ and introduce $\xi = \hat{u}_h - R_h\hat{u}$ and $\gamma =  R_h \hat{u} - \hat{u}$ so that
$e = \xi + \gamma$. For $\gamma$ we have the error estimate in Lemma \ref{lem:ritz-conv}. What remains is to find a bound for $\xi$ as well.
We have
\begin{equation*}
    (s^2(\xi+\gamma),v) + \sum_{\eta = x,y,z}(\frac{1}{\rho S_{\eta}^2}\pdd{(\xi+\gamma)}{\eta},\pdd{v}{\eta}) = 0.
\end{equation*}
We let $v= s\xi$ and get
\begin{equation*}
    s\|s\xi\|^2 + \sum_{\eta = x,y,z}(\frac{s^*S_{\eta}^*}{S_{\eta}\rho}\frac{1}{S_{\eta}}\pdd{\xi}{\eta},\frac{1}{S_{\eta}}\pdd{\xi}{\eta}) = -(s^2\gamma,s\xi)-\sum_{\eta = x,y,z}(\frac{s^*}{S_{\eta}^2\rho}\pdd{\gamma}{\eta},\pdd{\xi}{\eta}).
\end{equation*}
The last term disappear by the Ritz projection  \eqref{eq:ritz-like}. Adding conjugate to both sides yields
\begin{equation*}
    a(\|s\xi\|^2 +\|\boldsymbol{\nabla}_d \xi\|^2) \leq \|s^2 \gamma\| \|s \xi\|
\end{equation*}
which further gives us
\begin{equation}
    \|s\xi\| \leq \frac{1}{a}\|s^2\gamma\|\leq \frac{1}{a}\|s^2(R_h\hat{u}-\hat{u})\|\leq \frac{h^{k+1}}{a^3}\|s^4\hat{u}\|_{H^{k+1}}. 
\end{equation}
Note that $\|s\xi\|=\lvert s\rvert \|\xi\|$ and $\|s^4u\|_{H^{k+1}}=\lvert s\rvert \|s^3\hat{u}\|_{H^{k+1}}$. We now have our estimate of 
\begin{equation}\label{eq:xi_in}
    \|\xi\| \leq C\frac{h^{k+1}}{a^3}\|s^3\hat{u}\|_{H^{k+1}}
\end{equation}
Combining \eqref{eq:xi_in} and Lemma \ref{lem:ritz-conv} together with the triangle inequality we get our desired result.
\end{proof}
\begin{remark}Theorem \ref{thm:convergence} was proven for constant damping coefficients $d_{\eta}$. To show the corresponding convergence result for varying damping coefficients we need a discrete stability result, which is based on exact integration. As discussed above such a stability result is likely to be a perturbation of the stability estimate for the particular discretization discussed in Section 6.1. It is therefore not surprising that good convergence properties are seen in Section 7  for our discretizations of the variable damping case.
\end{remark}
\section{Numerical Experiments\label{sec:experiments}}
In this section, we present some numerical experiments to validate the theoretical results presented in Section \ref{sec:discrete}, as well as to quantify the numerical PML errors.

All the experiments in this section are performed in 2D on the domain $\Omega_T = [-6,6]\times [-6,6]$ with PML layer size $\Delta_{pml}  = 0.6$ along all sides. 
We denote the original domain, i.e. the domain without PML as $\Omega_{in} = [-5.4,5.4]\times[-5.4,5.4]$.
We use the spatial discretisation presented in Section \ref{sec:fem}.
In the experiments, we use standard continuous square quadrilateral Lagrange elements, $Q_p$, for the solution $u$, where $p = 1,2,3$ denotes the order of the Lagrange polynomials. To match the theory we use discontinuous square quadrilateral Lagrange elements for the auxiliary variables. Gauss-Legendre quadrature nodes are used for integration over the elements.
As boundary condition, we use the Dirichlet boundary condition defined in equation \eqref{eq:dirichlet}, which is imposed strongly.
The initial condition is set to zero and we have a forcing function in the form of a Gaussian pulse centered at $(0,0)$.
For all experiments, we let the damping profile be defined as 
\begin{align}\label{eq:dampingx}
d_x= \left\{ \begin{array}{cc} 
               d_0( \frac{\lvert x\rvert-5.4}{\Delta_{pml}})^3, & \hspace{5mm} if \lvert x \rvert \geq 5.4 \\
                0, & \hspace{5mm} otherwise \\
                \end{array} \right.
\end{align}
\begin{align}\label{eq:dampingy}
d_y= \left\{ \begin{array}{cc} 
               d_0( \frac{\lvert y\rvert-5.4}{\Delta_{pml}})^3, & \hspace{5mm} if \lvert y\rvert \geq 5.4 \\
                0, & \hspace{5mm} otherwise \\
                \end{array} \right.
\end{align}
where $d_0\geq0$ is the damping strength coefficient and it is chosen as in \cite{duru2016role} to balance the modeling error and the discretisation error. It is given by
\begin{equation}
    d_0 = \frac{4c}{2\Delta_{pml}}ln(\frac{1}{tol})
\end{equation}
where $\Delta_{pml}$ is the PML-width, c is the wave-speed and $tol$ is chosen as 
\begin{equation}
    tol = C_0\left[\frac{1}{\Delta_{pml}}\frac{h}{p+1}\right]^{p+1}.
\end{equation}
Here $p$ is the order of the polynomial approximation, $h$ is the element size, and $C_0 = 2$.
The PML error tolerance, \emph{ tol}, is controlled by the PML with $\Delta_{pml}$, the order of the polynomial approximation $p$, and the element size $h$. For a PML of finite thickness the PML error will converge to zero at the rate of the accuracy of the underlying numerical method \cite{duru2016role,duru2019energy,DURU2014757}.

For time integration we will use the fourth-order Runge-Kutta method.

\subsection{Homogeneous problem}
In this subsection we consider the homogeneous case i.e. $\kappa = \rho = 1$, and first investigate the numerical PML-error in maximum norm.
In order to measure the numerical PML-error, we simulate both the PML problem as described in this section, as well as a reference problem.
The reference problem is a problem defined on the larger domain $\Omega_r = [-12,12]\times[-12,12]$, on which we solve the wave equation without any PML.
The size of $\Omega_r$ is chosen so that the reflections from the outer boundary will not reach $\Omega_{in}$ during the simulation.
We use the same initial condition, the same source, the same time step, and the same spatial resolution.
The numerical PML-error is defined as the difference between the PML solution and the reference solution, and is calculated at every node in $\Omega_{in}$, and at every time step. 
The simulation is performed for $t\in [0,10]$. At the end of this time interval a reflected wave in the reference problem will just have reached the outer boundary of $\Omega_{in}$. 
The computations are done with Q1, Q2, and Q3 elements, as well as for different element sizes.
We use a time step $\Delta t = 0.01$ for all the elements and all spatial resolutions. 
This time step is chosen so that the dominant error for all resolutions and all orders will come from the space discretisation.

In Figures \ref{fig:max_pml_1}, \ref{fig:max_pml_2}, and \ref{fig:max_pml_3} we have plotted the maximum PML-error in $\Omega_{in}$ as a function of time for different resolutions and for different polynomial orders. In Figure \ref{fig:pml_conv} the convergence of the PML-error in maximum norm can be seen. The error is measured at the end time $t = 10$. We can see that the solution converges with at least the expected rates as we refine the element size.

To verify the stability results from Section \ref{sec:discrete} we look at the long-time behavior of our simulated PML-problem. 
For this purpose solve the same problem as before but for times $t\in[0,150]$, so that we can see what happens when the simulation is allowed to run for a longer time. 
We then measure the maximum amplitude of our solution in $\Omega_{in}$. The result can be seen in Figure \ref{fig:pml_lts}.
Note that no long-time growth of the solution can be observed.

In Figure \ref{fig:snap_homogeneous} we can see snapshots of the solution at times $t = 2$, $t = 4$, $t=6$ and $t = 12$. We can follow how the Gaussian pulse propagates in $\Omega_{in}$ and vanishes as it enters the PML.
\begin{figure}
\begin{subfigure}{.5\textwidth}
  \centering
  \includegraphics[width=0.4\paperwidth]{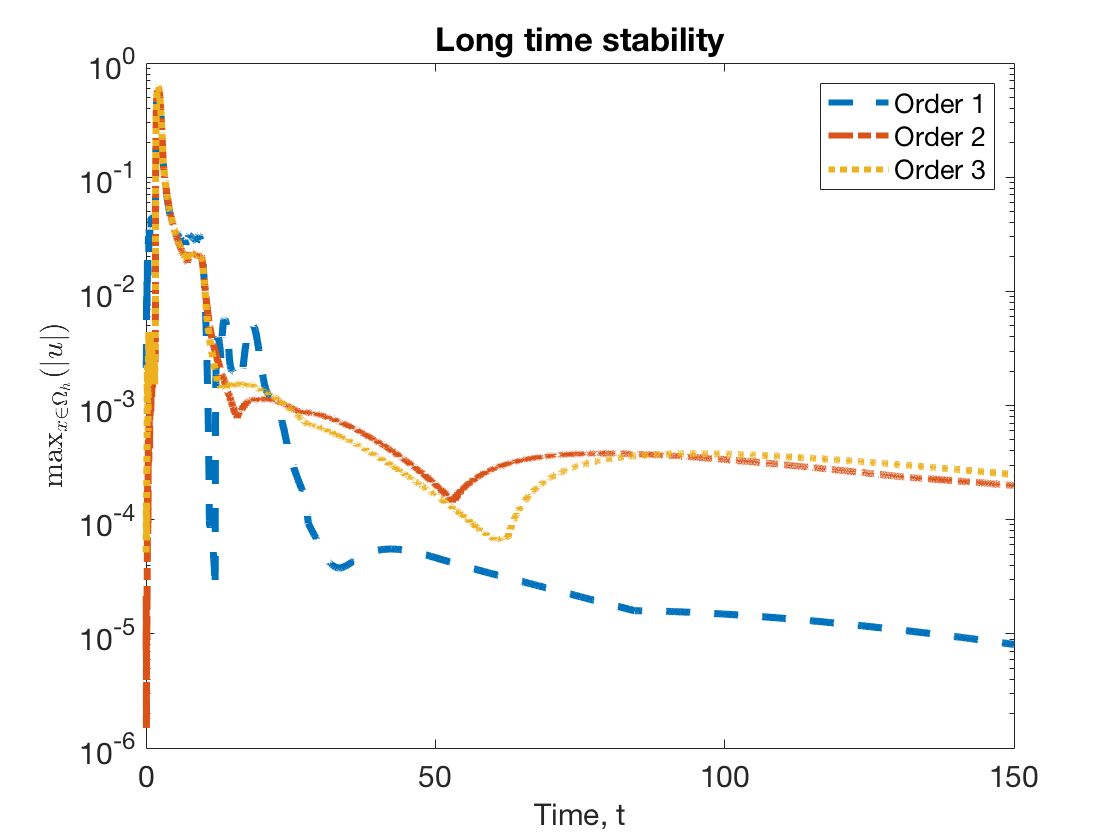}
  \caption{Homogeneous medium}
  \label{fig:pml_lts}
\end{subfigure}
\begin{subfigure}{.5\textwidth}
 \centering
  \includegraphics[width=0.4\paperwidth]{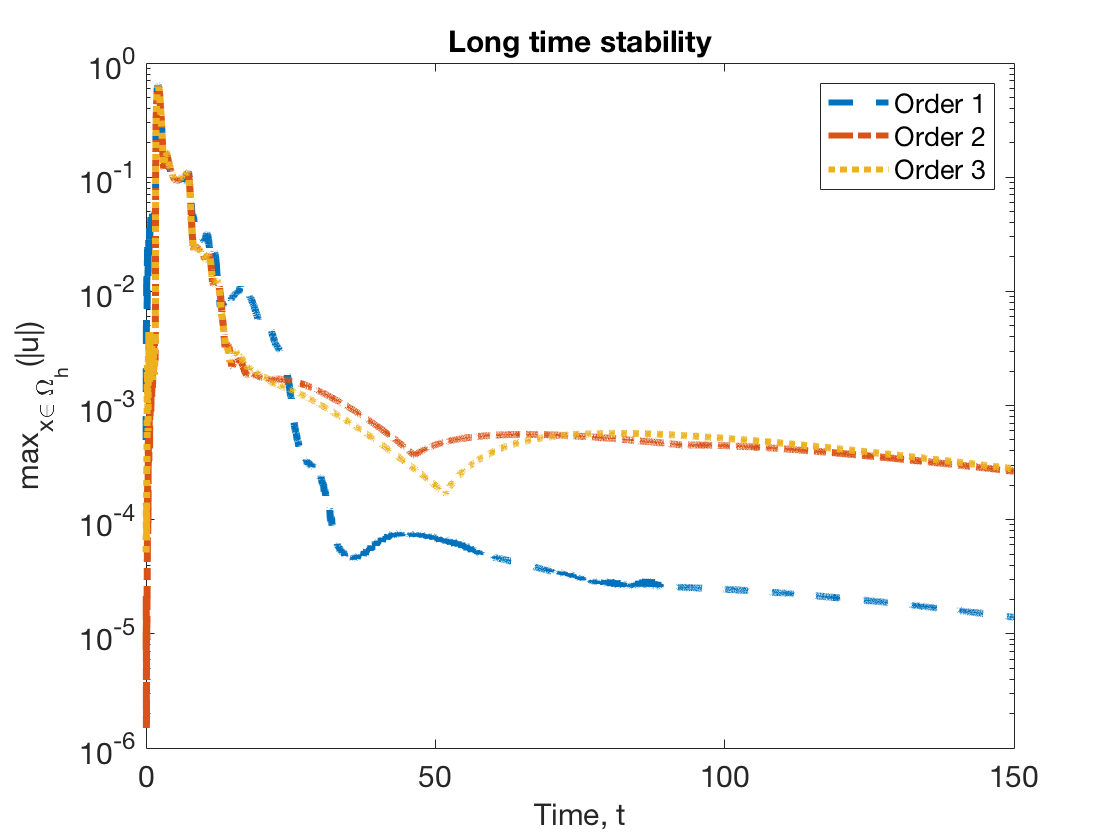}
  \caption{Heterogeneous medium
  }
  \label{fig:pml_lts_var}
\end{subfigure}
\caption{Long time stability for first, second and third order elements for the homogeneous and heterogeneous case.}
\label{fig:lts_both}
\end{figure}

\begin{figure}
\begin{subfigure}{.5\textwidth}
  \centering
  \includegraphics[width=.4\paperwidth]{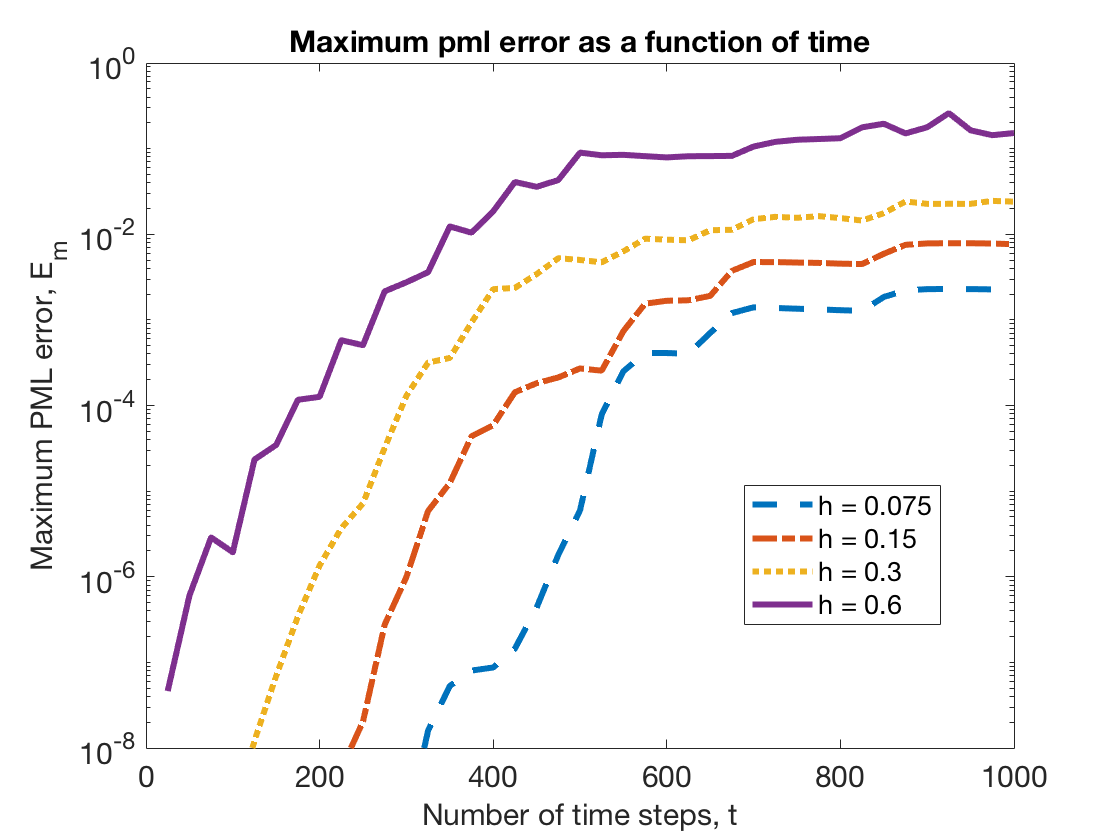}
  \caption{PML error using Q1 elements}
  \label{fig:max_pml_1}
\end{subfigure}
\begin{subfigure}{.5\textwidth}
  \centering
  \includegraphics[width=.4\paperwidth]{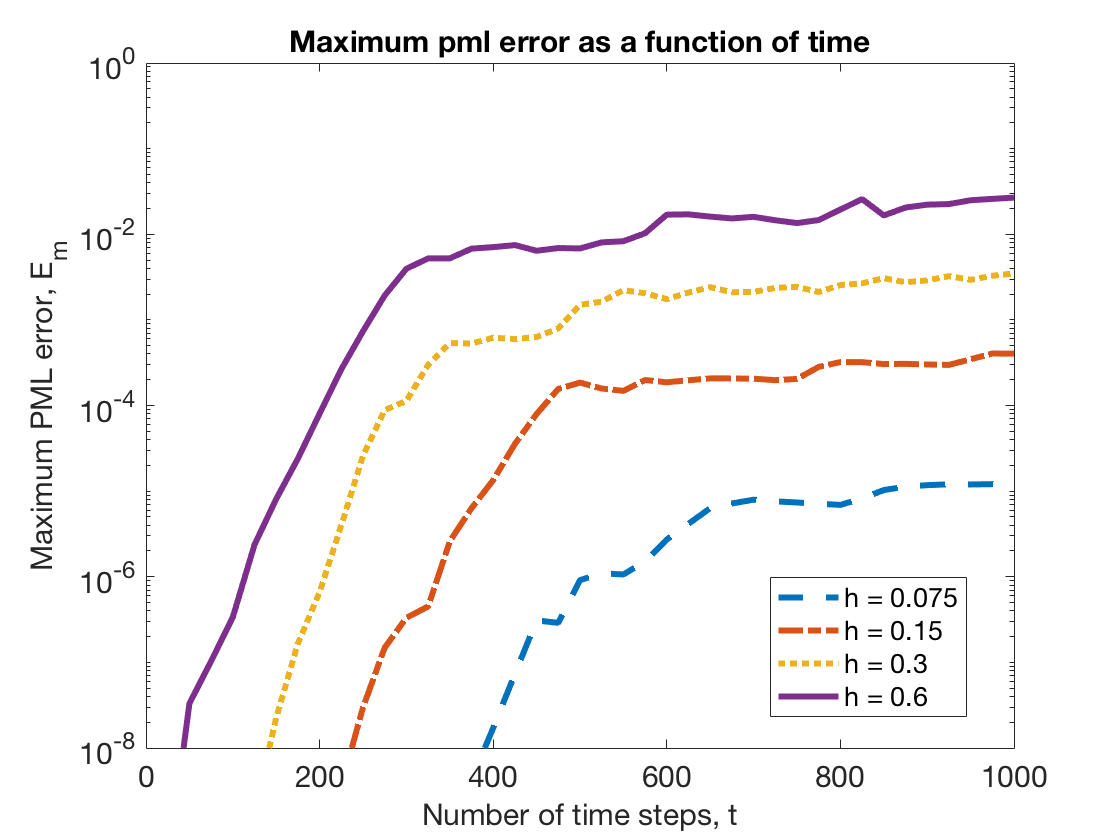}
  \caption{PML error using Q2 elements}
  \label{fig:max_pml_2}
\end{subfigure}
\newline
\begin{subfigure}{.5\textwidth}
 \centering
  \includegraphics[width=.4\paperwidth]{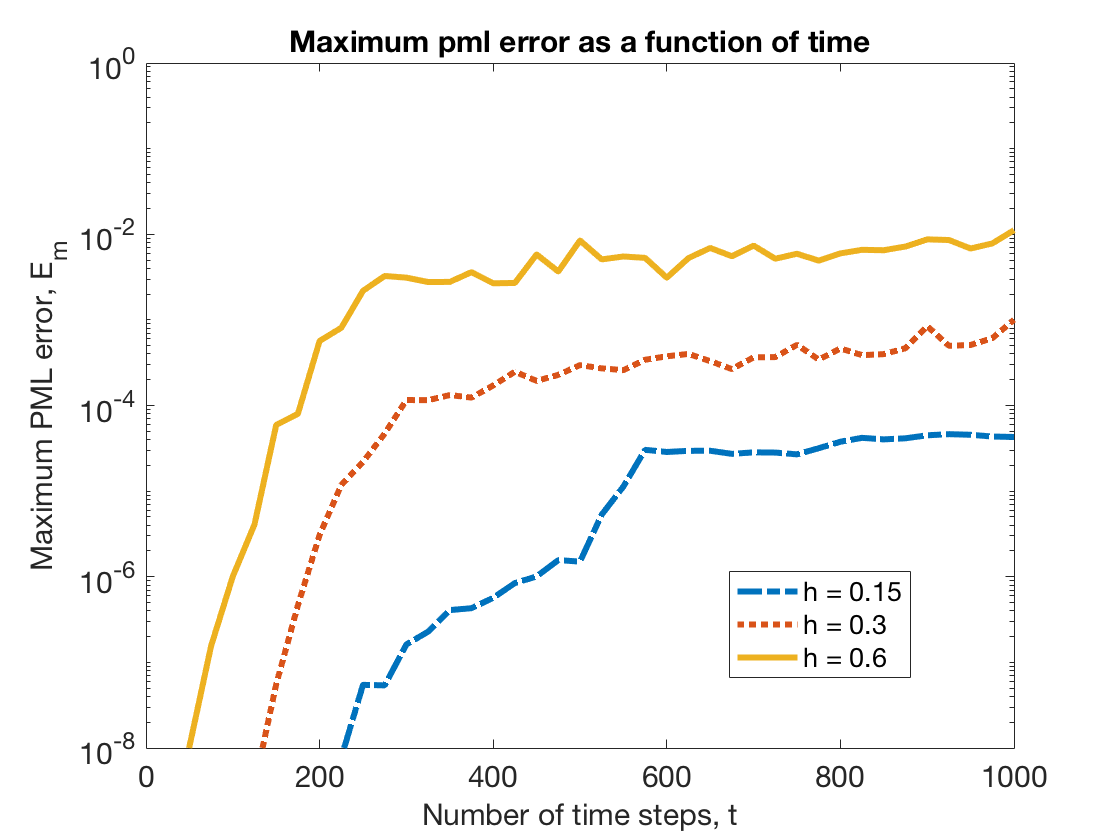}
  \caption{PML error using Q3 elements}
  \label{fig:max_pml_3}
\end{subfigure}
\begin{subfigure}{.5\textwidth}
  \centering
  \includegraphics[width=.4\paperwidth]{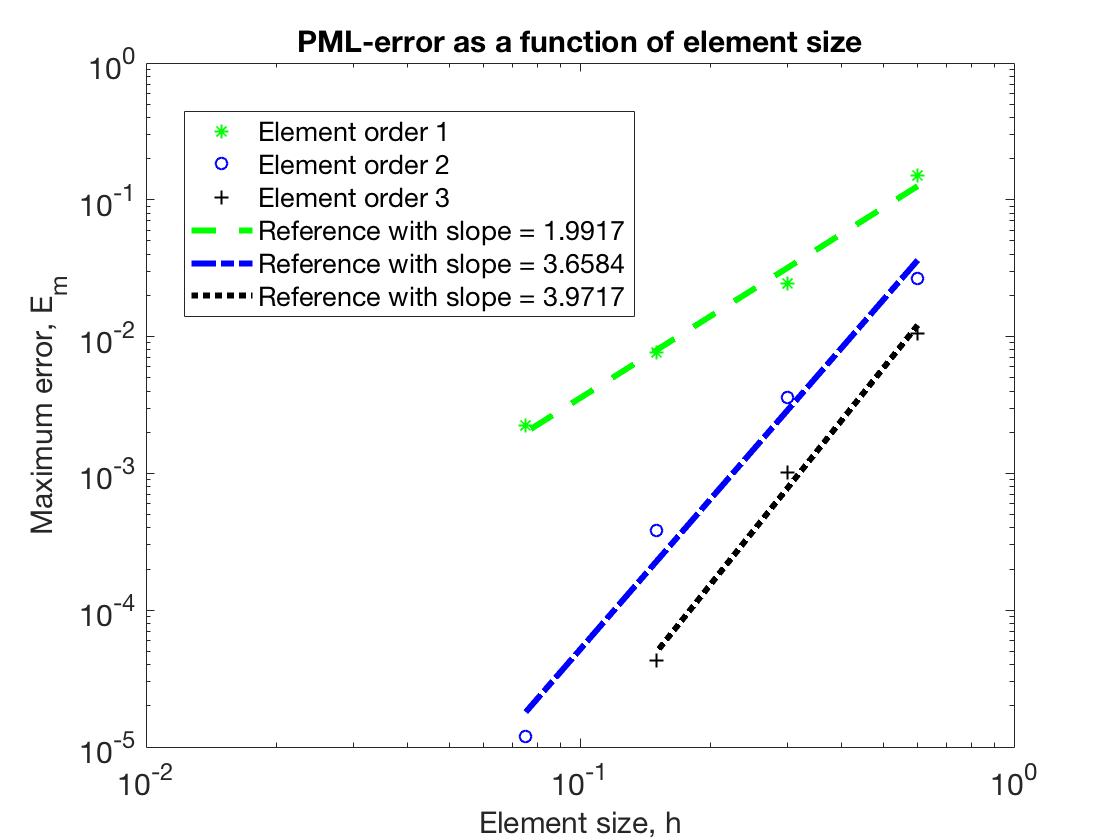}
  \caption{Convergence of the PML error}
  \label{fig:pml_conv}
\end{subfigure}
\caption{PML errors versus number of time steps can be seen in Figures \ref{fig:max_pml_1}, \ref{fig:max_pml_2}, \ref{fig:max_pml_3} for elements Q1, Q2 and Q3 respectively. In Figure \ref{fig:pml_conv} we can see the PML error convergence for Q1, Q2 and Q3 elements at end time $t = 10$}
\label{fig:homogeneous}
\end{figure}
\begin{figure}
\begin{subfigure}{.5\textwidth}
  \centering
  \includegraphics[width=1\linewidth]{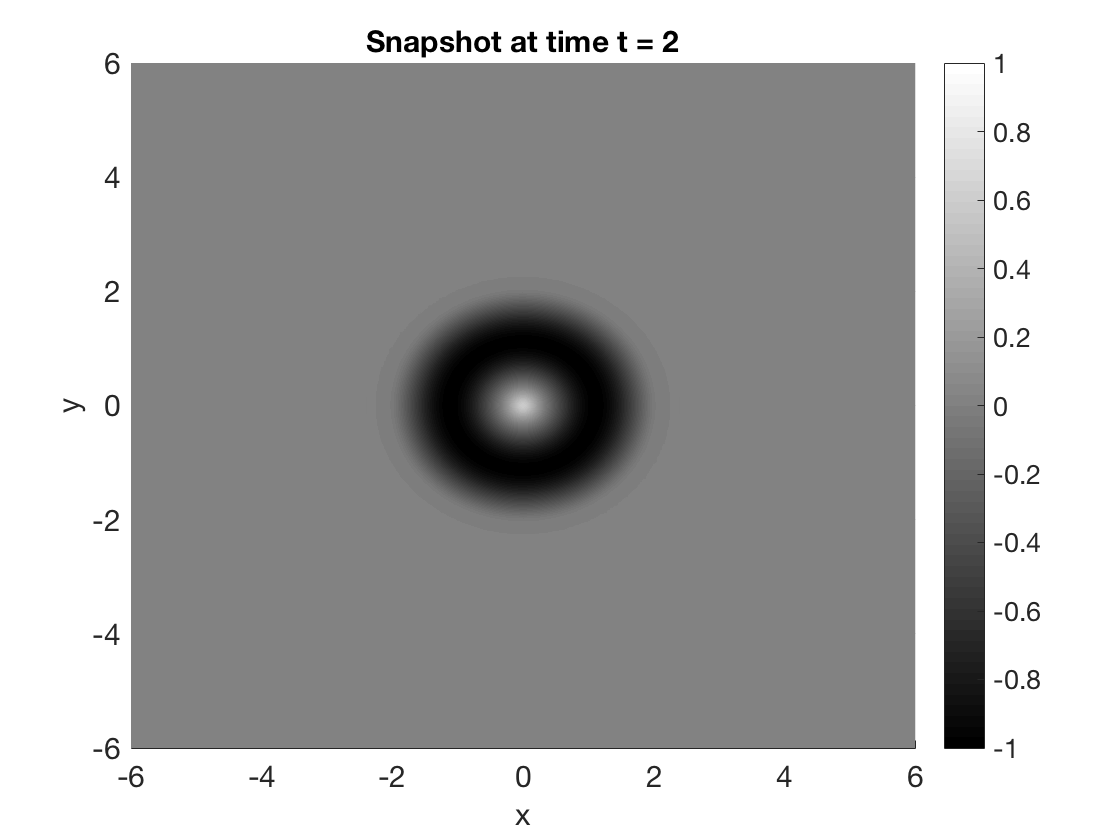}  
  \caption{Snapshot at $t = 2$}
  \label{fig:sub-first}
\end{subfigure}
\begin{subfigure}{.5\textwidth}
  \centering
  \includegraphics[width=1\linewidth]{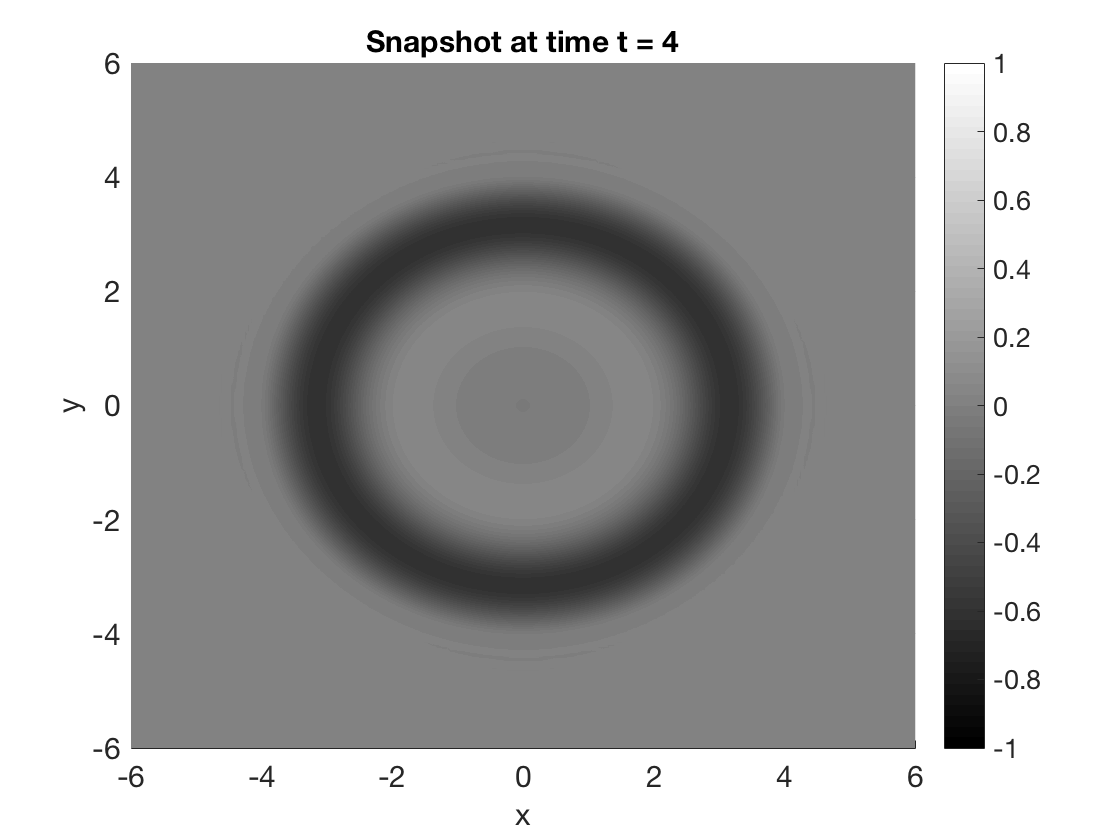}  
  \caption{Snapshot at $t = 4$}
  \label{fig:sub-second}
\end{subfigure}
\newline
\begin{subfigure}{.5\textwidth}
  \centering
  \includegraphics[width=1\linewidth]{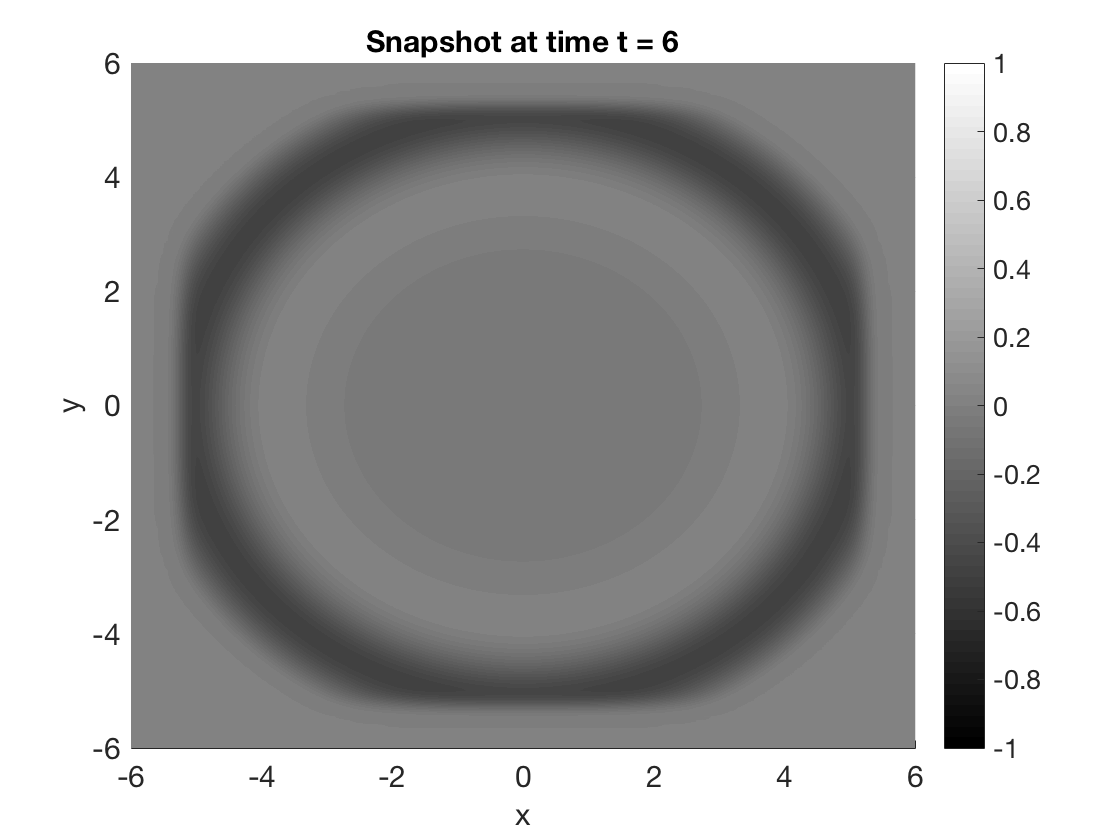}  
  \caption{Snapshot at $t = 6$}
  \label{fig:sub-third}
\end{subfigure}
\begin{subfigure}{.5\textwidth}
  \centering
  \includegraphics[width=1\linewidth]{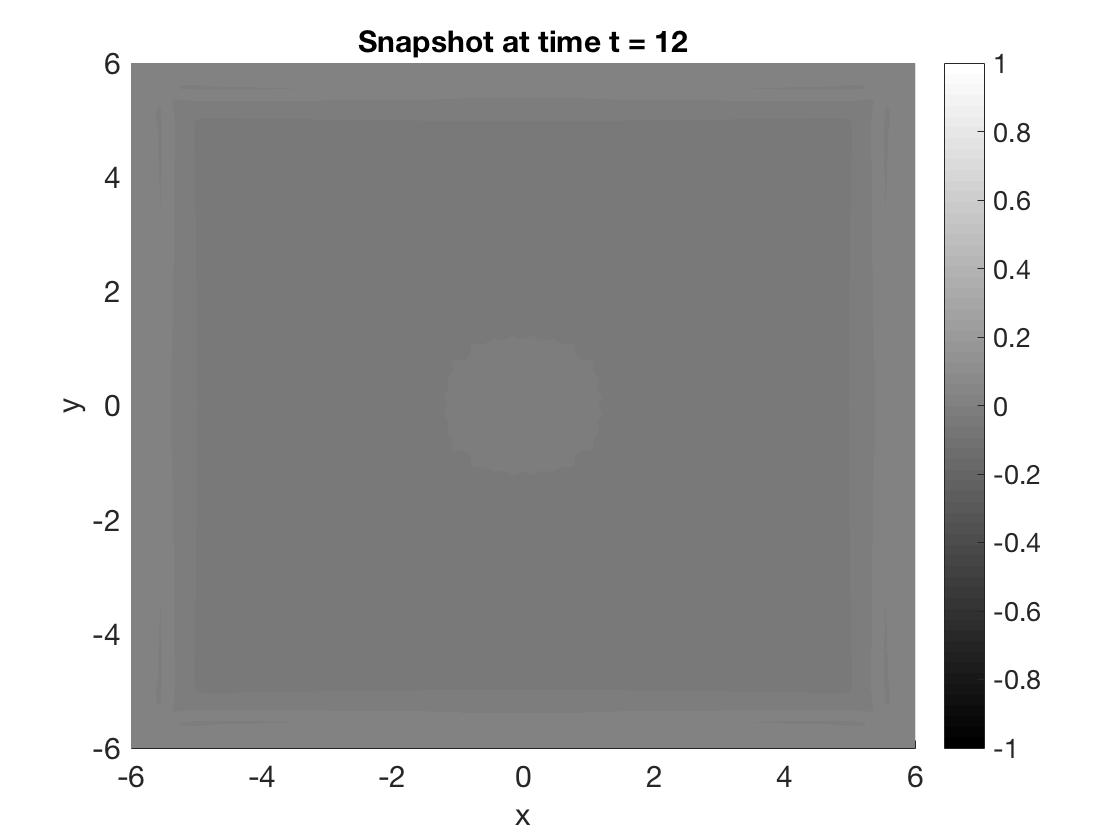}  
  \caption{Snapshot at $t = 12$}
  \label{fig:sub-fourth}
\end{subfigure}
\caption{Snapshots of the solution to the homogeneous case at different times using Q3 elements with resolution $h = 0.3$ and $dt = 0.01$}
\label{fig:snap_homogeneous}
\end{figure}

\subsection{Heterogeneous problem}
Next we consider the heterogeneous problem. The setup is the same as for the homogeneous problem except that we let the wave speed vary in the domain and we have chosen our end time to be $t = 14$ to make up for the lower wave speed. The wave speed is chosen as
\begin{align}\label{eq:wave_speed}
c(x,y)= \left\{ \begin{array}{cc} 
                0.75,& \hspace{5mm} y>2.4 \\
                1, & \hspace{5mm} |y|\leq 2.4. \\
                 1.25, & \hspace{5mm} y<-2.4 \\
                \end{array} \right .
\end{align}

Note that the element boundaries are chosen so that they aline with the material interfaces. 

We perform the same experiments for the heterogeneous case as for the homogeneous case. 
In Figures \ref{fig:max_pml_1_var}, \ref{fig:max_pml_2_var} and \ref{fig:max_pml_3_var} we can see the maximum PML error as a function of time for different resolutions and for different elements Q1, Q2 and Q3.

In Figure \ref{fig:pml_conv_var} we can see the convergence of the PML-error in maximum norm at the end time $t = 14$.
We can note that the convergence rate is slightly lower than expected for the Q3 elements.

In Figure \ref{fig:pml_lts_var} we can see the long-time behavior of the solution to the heterogeneous problem. 
To capture the long-time behavior we run the simulation until $ t = 150$, just as in the homogeneous case.
No asymptotic growth is observed, which again agrees with the stability analysis in Section \ref{sec:discrete}.
Finally, we can see snapshots of the solution at different times in Figure \ref{fig:snap_var}.
We can follow how the Gaussian pulse propagates and vanishes as it enters the PML. One can also see that the wave spreads faster when $y<-2.4$ and slower when $y>2.4$.

\begin{figure}
\begin{subfigure}{.5\textwidth}
  \centering
  \includegraphics[width=.4\paperwidth]{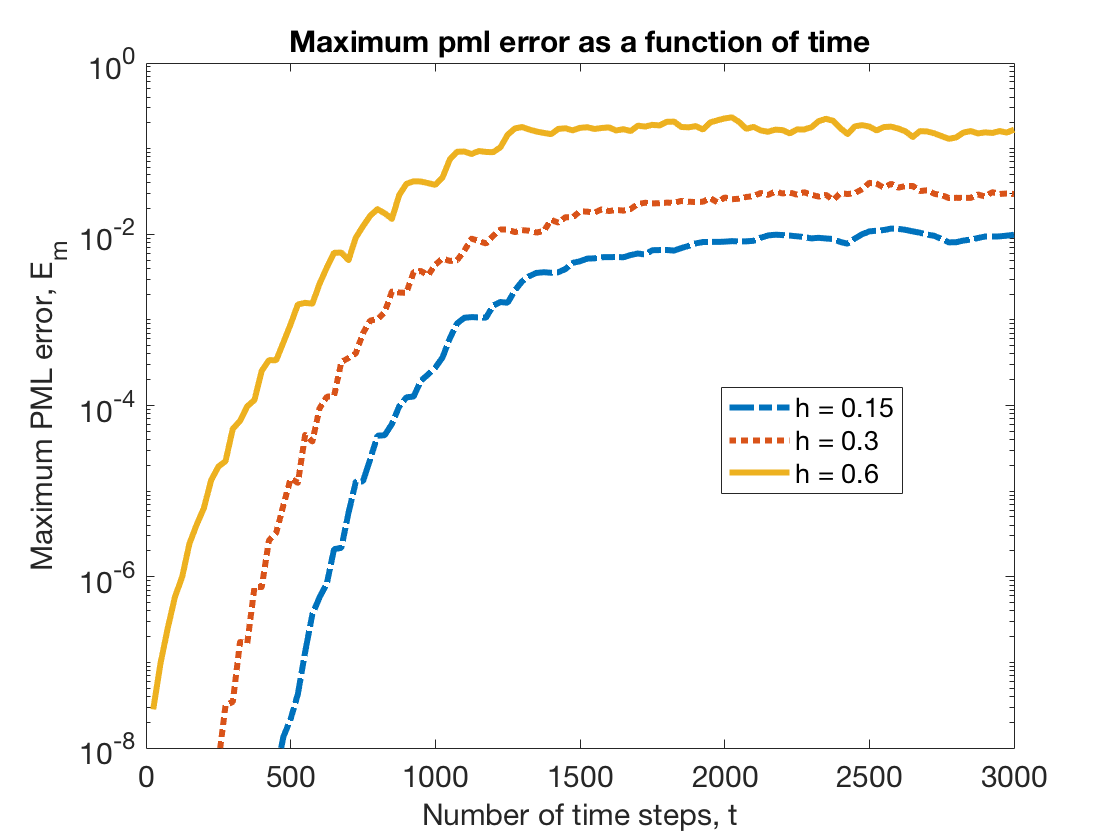}
  \caption{QPML error using Q1 elements}
  \label{fig:max_pml_1_var}
\end{subfigure}
\begin{subfigure}{.5\textwidth}

  \centering
  \includegraphics[width=.4\paperwidth]{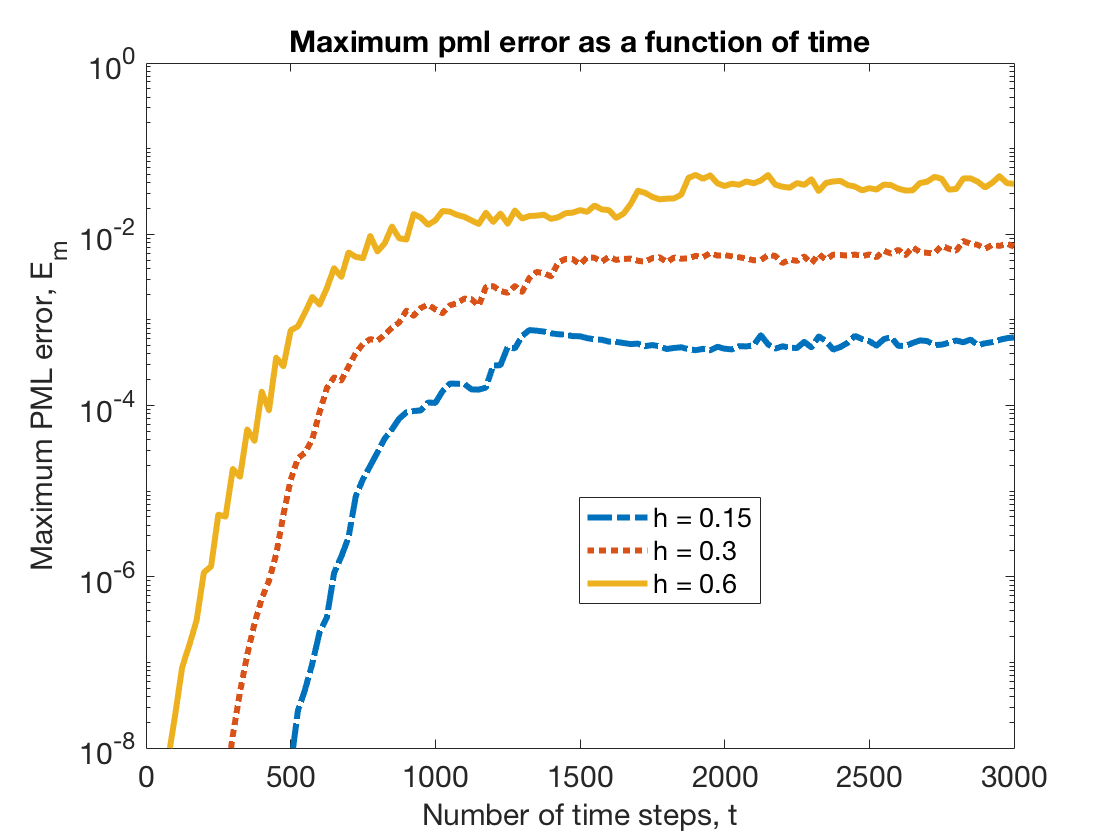}
  \caption{PML error using Q2 elements}
  \label{fig:max_pml_2_var}
\end{subfigure}
\newline

\begin{subfigure}{.5\textwidth}
  \centering
  
  \includegraphics[width=0.4\paperwidth]{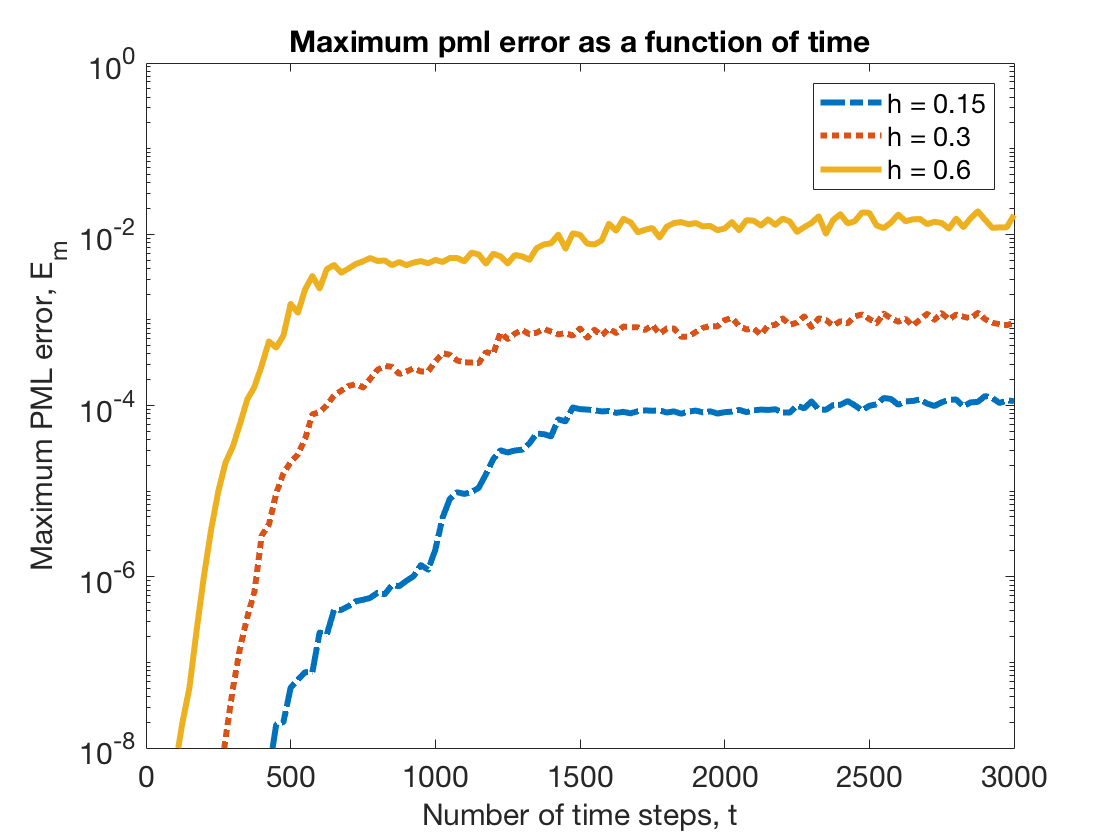}
  \caption{PML error using Q3 elements}
  \label{fig:max_pml_3_var}
\end{subfigure}
\begin{subfigure}{.5\textwidth}

  \centering
  \includegraphics[width=0.4\paperwidth]{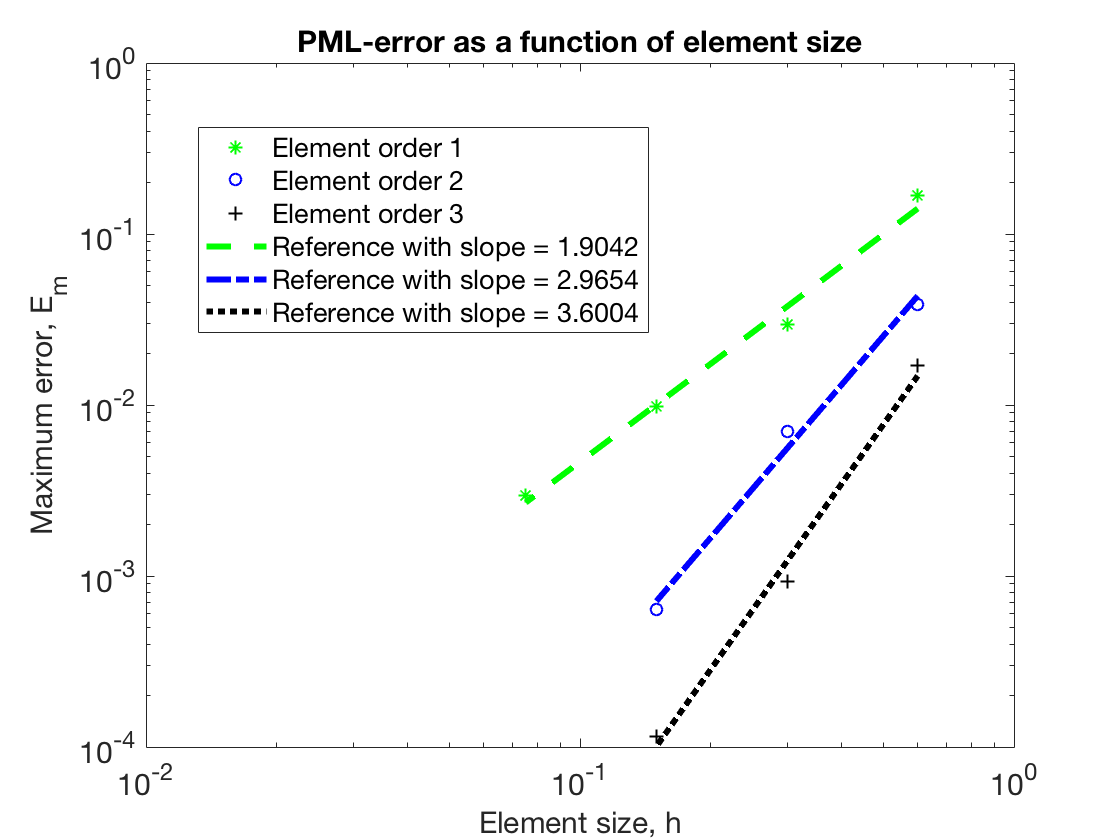}
  \caption{Convergence of the PML error in maximum norm.}
  \label{fig:pml_conv_var}
\end{subfigure}
\caption{PML errors versus number of time steps for different resolutions can be seen in Figures \ref{fig:max_pml_1_var}, \ref{fig:max_pml_2_var}, \ref{fig:max_pml_3_var} for elements Q1, Q2 and Q3 respectively.  In Figure \ref{fig:pml_conv} we can see the convergence of the PML error for Q1, Q2 and Q3 elements at end time $t = 14$ }
\label{fig:heterogeneous}
\end{figure}

\begin{figure}
\begin{subfigure}{.5\textwidth}
  \centering
  \includegraphics[width=1\linewidth]{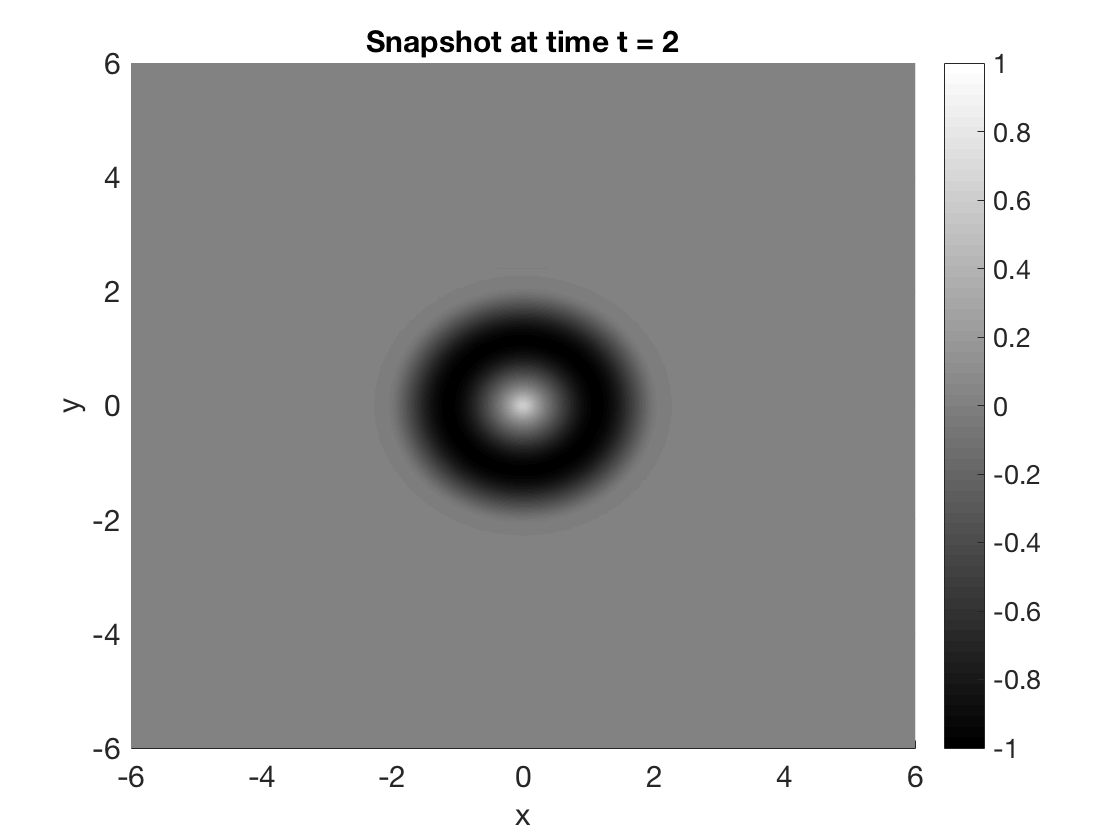}  
  \caption{Snapshot at $t = 2$}
  \label{fig:snap_var_1}
\end{subfigure}
\begin{subfigure}{.5\textwidth}
  \centering
  \includegraphics[width=1\linewidth]{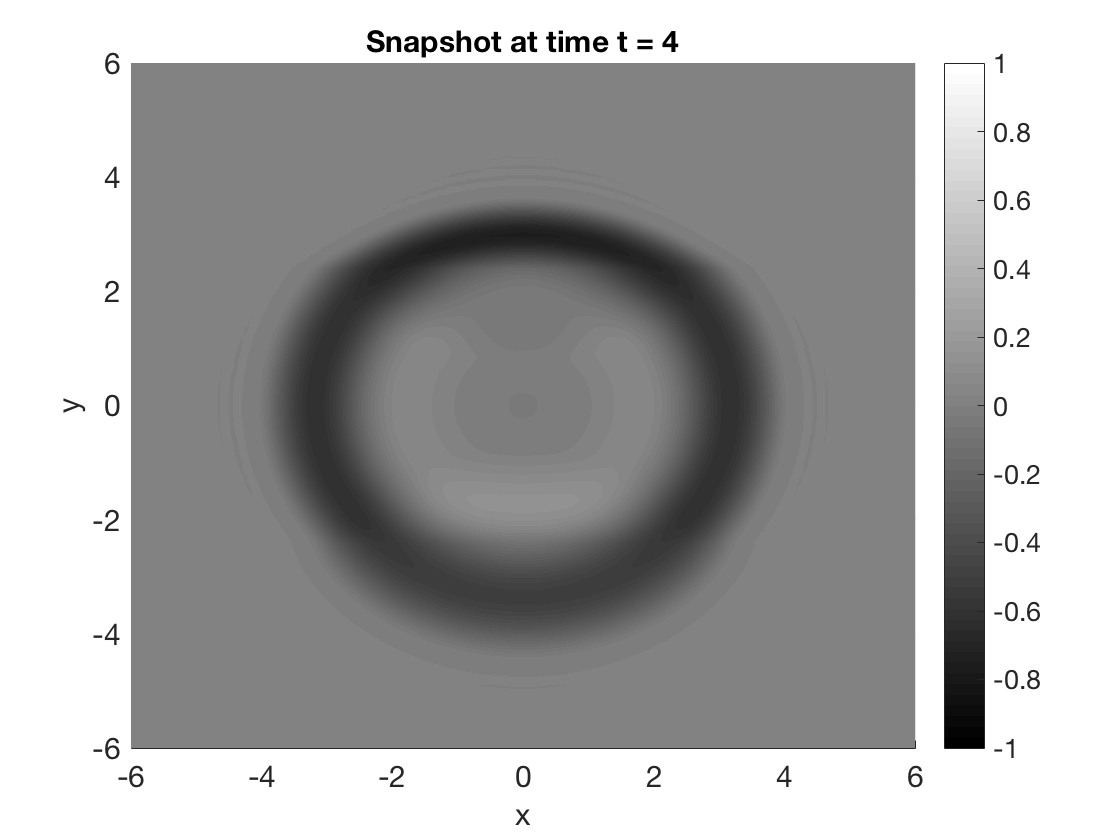}  
  \caption{Snapshot at $t = 4$}
  \label{fig:snap_var_2}
\end{subfigure}
\newline

\begin{subfigure}{.5\textwidth}
  \centering
  \includegraphics[width=1\linewidth]{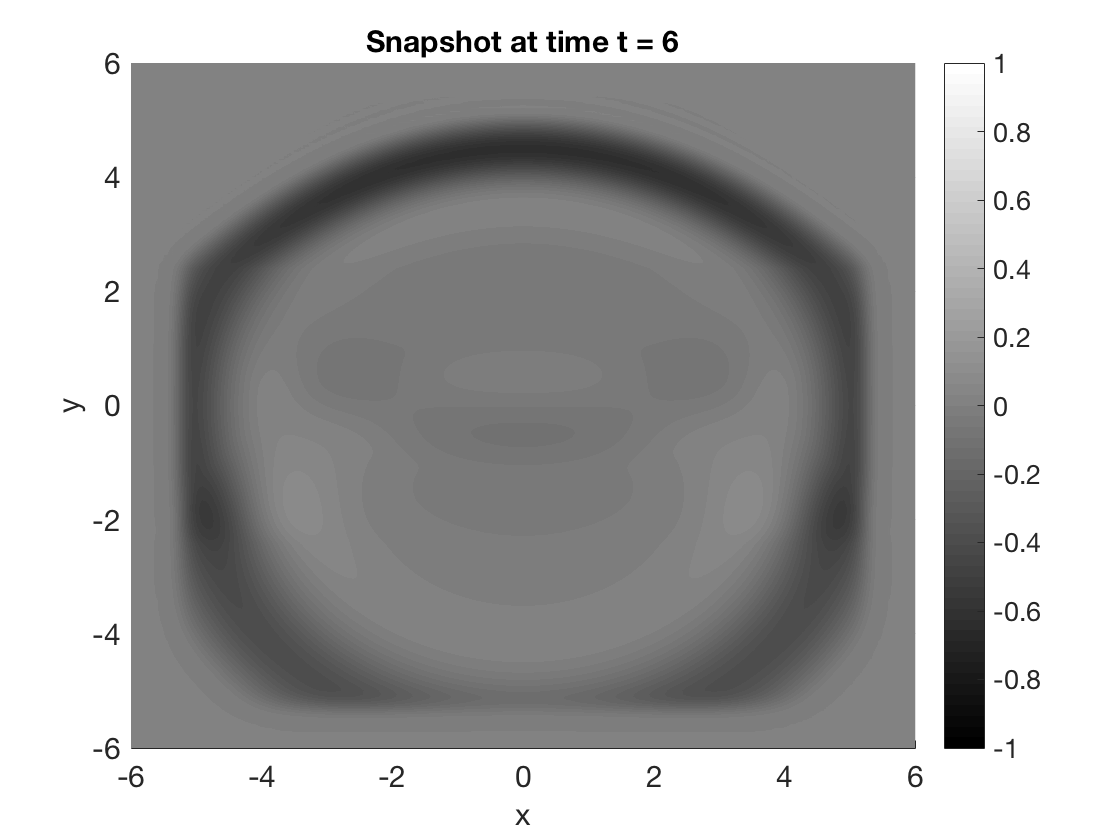}  
  \caption{Snapshot at $t = 6$}
  \label{fig:snap_var_3}
\end{subfigure}
\begin{subfigure}{.5\textwidth}
  \centering
  \includegraphics[width=1\linewidth]{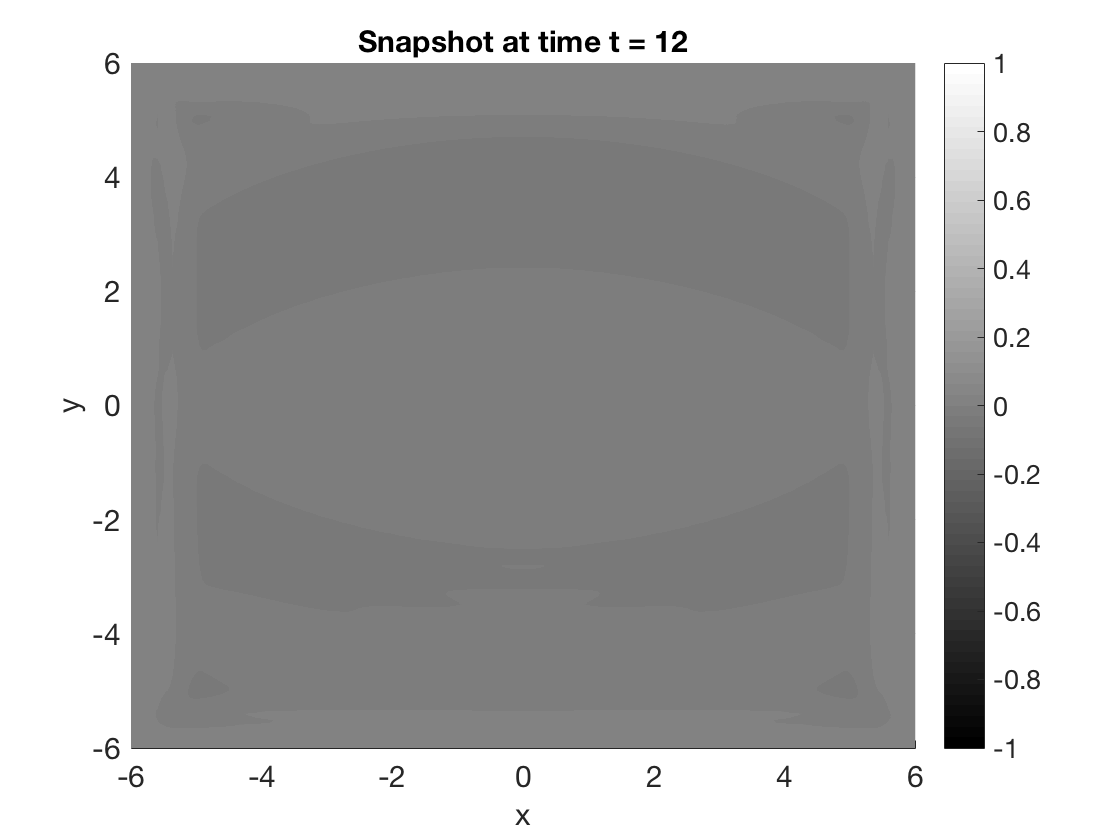}  
  \caption{Snapshot at $t = 12$}
  \label{fig:snap_var_4}
\end{subfigure}
\caption{Snapshots of the solution to the heterogeneous case at different times using Q3 elements with resolution $h = 0.3$ and $dt = 0.01$}
\label{fig:snap_var}
\end{figure}
\section{Summary and Discussion\label{sec:discussion}}
We have presented a stable FEM for the acoustic wave equation on second-order form with a PML. To begin with, we prove an energy estimate in Laplace space, which extends the result in \cite{duru2019energy} from constant damping to continuously varying damping. We also formulate the problem weakly in terms of a bi-linear form. In this formulation, we use the standard $H^1$ space for the physical variable and a combination of $H^1$ and $L_2$ spaces for the auxiliary variables. Our FEM is then obtained by restricting to the corresponding continuous and discontinuous piece-wise polynomial finite-dimensional subspaces. The chosen weak form allows us to mimic the energy estimate in a discrete setting.

An a priori error estimate for the case of constant damping follows by applying standard techniques in Laplace space. It indicates convergence of order $p+1$, where $p$ is the local polynomial order. For variable damping, a technical difficulty is related to the fact that the discrete stability relies on quadrature while exact integration is used in the continuous case. This prevents the straightforward extension of the convergence result to spatially varying damping. 

Numerical experiments for the case of varying damping verify the stability and convergence results. We compute in 2 space dimensions and use square, Cartesian elements. Note that the theoretical results presented in Section \ref{sec:discrete} are more general and cover the 3D problem with no hanging nodes as the only restriction on element shape. Therefore we are confident that when the method is applied to a 3D problem, it will be stable and have good accuracy properties.

Note that the method uses a discontinuous function space for some of the auxiliary variables. Experiments using continuous function spaces for all variables were also performed, but with less accurate results. The degradation of accuracy was pronounced for piece-wise linears, in analogy with results in \cite{cockburn2016static}. 

A standard approach to stability is to derive energy estimates in physical space. The analysis in this paper is done in Laplace space, both for the continuous and the discrete problems. It demonstrates how an energy estimate in Laplace space for the continuous problem can be transformed to stability and accuracy results for a discrete method. 
By performing analysis in Laplace space, the energy approach is applicable to many more problems, and it is our expectation that the techniques presented in this paper will be useful for many different problems.
\section*{Acknowledgments}
This research was partially supported by the Swedish Research Council.
\bibliographystyle{plain}
\bibliography{references}

\end{document}